          \def\version{4 August 2024}	 %

\documentclass[reqno,11pt]{amsart}  
\usepackage[utf8]{inputenc}
\pdfoutput=1

\usepackage[colorinlistoftodos]{todonotes}

\usepackage{amsmath}
\usepackage{amssymb} 
\usepackage{braket}
\usepackage{mathtools}   
\usepackage{graphicx}
\usepackage{graphics}  
\usepackage{color}
\usepackage{verbatim} 
\usepackage[normalem]{ulem} 
\usepackage[mathscr]{eucal}
\usepackage{upgreek}
\usepackage{enumerate} 
\usepackage{cite}
\usepackage{amsmath}
\usepackage{amsfonts}
\usepackage{amssymb}
\usepackage{bbm}
\usepackage{cancel}
\usepackage{graphicx}
\usepackage{multicol}
\usepackage[utf8]{inputenc}
\usepackage{framed}
\usepackage{setspace}
\usepackage{amsthm}
\usepackage{hyperref}
\usepackage{caption}
\usepackage{subcaption}
\usepackage{bbm}
\usepackage{graphicx}
\usepackage[titletoc]{appendix}
\numberwithin{equation}{section}  
 
\usepackage[refpage,noprefix]{nomencl}
\usepackage{nomencl}

\newcommand{\gk}[1]{\left\{#1\right\}}

\newcommand{\ek}[1]{\left[#1\right]}
\newcommand{\rk}[1]{\left(#1\right)}

\newcommand{\abs}[1]{\left| #1 \right|}

\newcommand{\Ccal}   {{\mathcal C }}

\newcommand{\Hcal}   {{\mathcal H }}

\newcommand{\Ocal}   {{\mathcal O }}

\newcommand{\Vcal}   {{\mathcal V }}

\newcommand{\BB}{\xi} 

\newcommand{\R}     {\mathbb{R}} 
\newcommand{\Z}     {\mathbb{Z}} 
\newcommand{\N}     {\mathbb{N}} 
\renewcommand{\P}   {\mathbb{P}} 
 
\newcommand{\E}     {\mathbb{E}}

 \newcommand{\ex}{{\rm e}} 
 \newcommand{\eps}{{\varepsilon}} 
 \renewcommand{\d}{{\rm d}}

\renewcommand{\l}{\lambda}
\newcommand{\e}{\varepsilon}
\newcommand{\1}{\mathbbm{1}}

\newcommand{\QV}[1]{\textcolor{red}{QV: #1}}

\newcommand{\WK}[1]{\textcolor{blue}{WK: #1}}

\renewcommand{\P}{\mathbb{P}}

\newcommand{\Pfrak}{\mathfrak{P}}
\newcommand{\Nrd}{{\mathfrak N}}

\newtheorem{theorem}{Theorem}[section]

\newtheorem{proposition}[theorem]{Proposition}
\newtheorem{lemma}[theorem]{Lemma}

\newtheorem{corollary}[theorem]{Corollary}

\newtheoremstyle{rem}{1.5ex}{1.5ex}{\rmfamily}{} {\bfseries\rmfamily}{} {1.5ex}{}
\theoremstyle{rem}
\newtheorem{remark}[theorem]{{\bfseries Remark}}

\newcommand{\Ns}{{\mathfrak N}^{\ssup{\mathrm{short}}}}
\newcommand{\Nl}{{\mathfrak N}^{\ssup{\mathrm{long}}}}

\newcommand{\norm}[1]{\left|\hspace{-0.395mm}\left| #1\right|\hspace{-0.395mm}\right|}

\newcommand{\Ett}{{\tt E}}
\newcommand{\Ptt}{{\tt P}}
\newcommand{\oldgamma}{\mathfrak W}

\newcommand{\smfrac}[2]{\textstyle{\frac {#1}{#2}}}

\newcommand{\ssup}[1] {{\scriptscriptstyle{({#1}})}}

\newcommand{\Wj}{\mathcal{W}}
\newcommand{\Var}{\mathtt{Var}}

\title[ODLRO in a mean-field Bose gas via the Feynman--Kac formula]{Proof of off-diagonal long-range order\\\medskip in a mean-field trapped Bose gas\\\medskip via the Feynman--Kac formula}
\author[\qquad \hfill Bai, K\"onig, Vogel]{Tianyi Bai$^{1}$, Wolfgang K\"onig$^{2,3}$, and Quirin Vogel$^{4}$}

\vspace{2mm}

\begin{document}

\maketitle

\centerline{\version}

\bigskip

\centerline{\textit{$^1$Chinese Academy of Sciences, Beijing, China}}
\centerline{\textit{$^2$Weierstra\ss-Institut f\"ur Angewandte Analysis und Stochastik, Berlin, Germany}}
\centerline{\textit{$^3$Institut f\"ur Mathematik, Technische Universit\"at Berlin, Berlin, Germany}}
\centerline{\textit{$^4$Ludwig-Maximilians-Universität München,
Mathematisches Institut, München, Germany}}

\begin{abstract}
We consider the non-interacting Bose gas  of $N$ bosons in dimension $d\geq 3$ in a trap in a mean-field setting with a vanishing factor $a_N$ in front of the kinetic energy. The choice $a_N=N^{-2/d}$ is the semi-classical setting and was analysed in great detail in a special, interacting case in \cite{DS21}. Using a version of the well-known Feynman--Kac representation and a further representation in terms of a Poisson point process, we derive precise asymptotics for the reduced one-particle density matrix, implying off-diagonal long-range order (ODLRO, a well-known criterion for Bose--Einstein condensation) for $a_N$ above a certain threshold and non-occurrence of ODLRO for $a_N$ below that threshold. In particular, we relate the condensate and its total mass to the amount of particles in long loops in the Feynman--Kac formula, the order parameter that Feynman suggested in \cite{F53}. For $a_N\ll N^{-2/d}$, we prove that all loops have the minimal length one, and  for $a_N\gg N^{-2/d}$ we prove 100 percent condensation and identify the distribution of the long-loop lengths as the Poisson--Dirichlet distribution.
\end{abstract}



\setcounter{section}{0} 

\section{Introduction and main results}\label{sec-Intro}

\noindent This work is a contribution to the condensation theory of the Bose gas. Our main objectives are the following.
\begin{itemize}
\item Derive new and physically relevant results on Bose condensation for a particular mean-field version,

\item rigorously give evidence for the strong relation between the condensate and the long loops in the famous Feynman--Kac representation of the gas,

\item provide new, probabilistic proofs and use the language and toolbox of probability, in order to attract also this community to this fascinating subject.
\end{itemize} 

Since the vague suggestion of Feynman \cite{F53} that the number of particles in long loops might be a relevant order parameter for describing the famous phenomenon of Bose--Einstein condensation, the Bose gas became popular also in the probability world as a mathematically interesting object to study. However, there are not many probabilistic investigations yet with real physical relevance, but the tendencies often went to creations of new probabilistic models and new questions. Here we concentrate on physically relevant questions, yet establishing and further pushing a probabilistic toolbox.

For the study of the condensation phase transition in the Bose gas, the most acknowledged, crucial object to study is the {\em reduced one-particle density matrix}, and the most important goal here is to prove that it shows {\em off-diagonal long-range order (ODLRO)}, which is generally acknowledged as a signal of {\em Bose--Einstein condensation (BEC)}. This is what we are going to do in this work for a particular version of the Bose gas. 

In our recent work \cite{KVZ23}, we did this for the standard free (i.e., non-interacting) Bose gas in the thermodynamic regime. The precise model that we consider here is a mean-field model in a fixed trap with a vanishing factor $a_N$ in front of the kinetic energy. For the  particular value $a_N=N^{-2/d}$, we are in the semi-classical setting, and this model is particularly interesting since it shows a condensation phase transition at a fixed positive temperature. This has been shown in \cite{DS21} in a special case and is under work in \cite{BK27} in more generality (but, however, without proof of ODLRO). The present paper shows the existence and absence of ODLRO for many other choices of $a_N$. Furthermore, we also give a description of the condensate as the total mass of particles in long loops in the Feynman--Kac formula, and an explicit identification of the limiting distribution of the lengths of the long loops.

\subsection{A mean-field Bose gas}\label{sec-Bosegas}
\noindent We consider a canonical bosonic non-interacting system of $N$ particles in a confining potential in $\R^d$. The corresponding Hamilton operator is given as
\begin{equation}\label{Hdef}
\Hcal_{a, w}^{\ssup N}=-a \sum_{i=1}^N \Delta_i+\sum_{i=1}^Nw(x_i),\qquad x_1,\dots,x_N\in\R^d,
\end{equation}
the $N$-particle operator with kinetic energy given by $a\in(0,\infty)$ times the standard Laplace operator in a confining (or trapping) potential $w\colon\R^d\to[0,\infty)$. The quantity $1/a$ is interpreted as the mass of the particles. We are under the usual assumptions that $w$ is bounded from below and, for simplicity, is continuous and explodes quickly enough to $\infty$ far out. Our precise assumptions are formulated at the beginning of Section~\ref{sec-results}.


We are interested in {\em bosons} and introduce a symmetrisation, i.e., we project on the set of symmetric, i.e., permutation invariant, wave functions. Furthermore, we consider the particle system at positive temperature $1/\beta\in(0,\infty)$. That is,  we consider the following trace:
\begin{equation}\label{Zdef}
Z_N(\beta,a,w)={\rm Tr}_{+}\big(\ex^{-\beta\Hcal_{a,w}^{\ssup N}}\big)=Z_N(\beta a,1,\smfrac 1aw),
\end{equation}
where the index $+$ denotes the symmetrisation, i.e., the application of the projection operator on the set of all permutation invariant functions. The quantity $Z_N(\beta,a,w)$ is called the {\em partition function} of the system. In this paper, we study a {\em mean-field regime}, where we  do not introduce any $N$-dependence in $w$. Instead, we pick the parameter $a=a_N$ depending on $N$. Indeed, we will assume that $(a_N)_{N\in\N}$ is bounded, and \begin{equation}\label{chidef}
\chi=\lim_{N\to\infty} N a_N^{d/2}\in[0,\infty]
\text{ exists.}
\end{equation}

We  will investigate the {\it limiting free energy,}
\begin{equation}\label{freeenergyMF}
f_{\rm MF}(\beta,\chi)=-\frac 1\beta\lim_{N\to\infty}\frac 1N\log Z_N(\beta, a_N,w),
\end{equation}
and the {\it one-particle reduced density matrix} $\gamma^{\ssup a}_N\colon \R^d\times\R^d\to[0,\infty)$  of the state 
$$
\Gamma^{\ssup a}_N=\frac 1 {Z_N(\beta,a_N,w)}\ex^{-\beta \Hcal^{\ssup N}_{a,w}},
$$ 
that is, 
\begin{equation}\label{gammaNdef}
\gamma^{\ssup a}_N(x,y)=N\int_{(\R^d)^{N-1}}\Gamma^{\ssup a}_N\big(x,x_2,\dots,x_N,y,x_2,\dots,x_N\big)\,\d x_2\cdots \d x_N,\qquad x,y\in\R^d\, ,
\end{equation}
where we used the symbol $\Gamma^{\ssup a}_N$ for both the operator and its kernel.

The principal $L^2(\R^d)$-eigenvalue of $\Gamma_N^{\ssup a}$ is defined as
\begin{equation}\label{Gammaeigenvalue}
\sigma_N^{\ssup a}=\sup_{f\in L^2(\R^d)\colon \|f\|_2=1}\langle f, \gamma^{\ssup a}_N f\rangle.
\end{equation}
We say that the Bose gas exhibits {\it off-diagonal long-range order (ODLRO)} if $\sigma_N^{\ssup a}$ is of order $N$ as $N\to\infty$. The occurrence of ODLRO is generally acknowledged (see \cite{LSSY05}) as a criterion for the occurrence of Bose--Einstein condensation (BEC).

The case $a=N^{-2/d}$ is particularly interesting and is called the {\em semi-classical limit} setting, see \cite{DS21} and \cite{BK27}. On this scale, the famous condensation phase transition is observed at a critical value $\beta_{\rm c}\in(0,\infty)$ of $\beta$. This has been first explored in \cite{DS21} and has been explicitly worked out in the special case $d=3$, $w(x)=\omega |x|^2$ and under the assumption that the Hessian matrix of $v$ satisfies a particular upper bound that depends on $\omega$. In this case, the condensation effect was proved to hold both on the level of a non-analyticity of the limiting free energy and in terms of ODLRO at a particular value of $\beta$. \cite{DS21} followed an approach that is very common in mathematical physics, via an energy-entropy description and a transition to the Fourier world, while \cite{BK27}, like the present paper, applies the Feynman--Kac formula, a Poisson-point process description and large-deviations techniques, to express and analyse a variational expression for the limiting free energy.

The goal of the present paper is two-fold: (1) we handle also the two cases of {\em sub-} and {\em super-semiclassical regime} (that is, $a_N\ll N^{-2/d}$ and  $a_N\gg N^{-2/d}$, respectively) and prove that  ODLRO does not hold, respectively does hold, and (2) we follow a  probabilistic route that relies on the well-known Feynman--Kac formula and a representation in terms of a Poisson point process, like in our recent paper \cite{KVZ23}. However, in this paper we handle only the non-interacting case and leave the general case to future work.

Let us remark that the special case $a_N=1$ has been considered already in \cite{AK08}, where the Feynman--Kac formula and a combinatorial large-deviations principle was applied to find a variational formula for the limiting free energy; they also provide evidence on 100 percent condensation, but this was not anymore deepened.

\subsection{Representation via  a Poisson point process}\label{sec-PPPintro}

It is the starting point of this paper that the partition function and density matrix can be written in terms of a crucial Poisson point process (PPP). This process was introduced to the study of the Bose gas in \cite{ACK10}, but was already used for the study of other phenomena in statistical mechanics (e.g., conformal invariance in dimension two) in \cite{LW04} under the name {\em Brownian loop soup}. Here we rely on the recent adaptation in \cite{KVZ23} and refer proofs to Appendix A there.

The {\em canonical Brownian bridge measure} from $x\in\R^d$ to $y\in\R^d$ with time horizon $\beta\in(0,\infty)$ is defined  by 
\begin{equation}\label{nnBBM}
\BB^{\ssup \beta}_{x,y}(A)=\frac{\P_x(B\in A;B_\beta\in\d y)}{\d y},\qquad A\subset\Ccal_{\beta}\mbox{ measurable},
\end{equation}
where $\Ccal_\beta$ denotes the set of all continuous functions $[0,\beta]\to\R^d$.
Here, $B=(B_t)_{t\in[0,\beta]}$ is a Brownian motion in $\R^d$ with generator $\Delta$, starting from $x$ under $\P_x$.  We introduce an integrated and weighted version on loops of $\Ccal_\beta$,
\begin{equation}\label{Vdef}
\BB^{\ssup{\beta,w}}( \d f)=\ex^{- \int_0^\beta w(f(s))\,\d s}\,\int_{\R^d} \d x\,\BB^{\ssup{\beta}}_{x,x}(\d f)
= \int_{\R^d} \d x\,\E_x\big[\ex^{- \int_0^\beta w(B_s)\,\d s}\1\{B\in \d f\}\1\{B_\beta\in\d x\}\big]/\d x.
\end{equation}

Now we introduce the Poisson process called the {\it Brownian loop soup}, the natural reference measure for the Feynman--Kac representation of the Bose gas. We write ${\tt P}_{\beta,w}^{\ssup N}$ for the probability measure of a Poisson point process (PPP) $\eta=\sum_f \delta_f$  with intensity measure 
\begin{equation}\label{nudef}
\nu_{\beta, w}^{\ssup N}(\d f)=\sum_{k=1}^N \frac 1k \BB^{\ssup{k \beta, w}}(\d f),\qquad f\in\bigcup_{k\in\N}\Ccal_{k\beta}=\widehat  \Ccal_\beta.
\end{equation}
If $f\in\Ccal_{k\beta}$ is an element of $\eta$, we say it is a \emph{loop} of length $\ell(f)=k$. Then $X_k=\#\{f\in\eta\colon \ell(f)=k\}$ is the number of loops with length $k$. Then $(X_k)_{k\in[N]}$ is a collection of independent Poisson-distributed random variables $X_k$ with parameter $\frac 1k \BB^{\ssup{\beta k,w}}(\Ccal_{\beta k})$. (We write $[N]$ for $\{1,\dots,N\}$.) We write ${\mathfrak N}(\eta)=\sum_{f\in \eta} \ell(f)=\sum_{k\in[N]}k X_k$ for the number of all particles  in the process $\eta$.  

The following is a variant of \cite[Lem. 1.2 and Cor. 1.4]{KVZ23}.

\begin{lemma}[PPP-representation of the reduced density matrix]\label{Cor-PPPreprfree}
For any $a$, $\beta\in(0,\infty)$ and $N\in \N$ and for all $x,y\in\R^d$,
\begin{equation}\label{PPPreprfree}
    \gamma^{\ssup a}_N(x,y)=\sum_{r=1}^N \BB_{x,y}^{\ssup{\beta a r,w/a}}(\Ccal_{\beta a r})
    \frac{{\tt P}_{\beta a,w/a}^{\ssup N}(\Nrd=N-r)}
    { {\tt P}_{\beta a,w/a}^{\ssup N}(\Nrd=N)}.
    \end{equation}
\end{lemma}

We refer to \cite[Appendix A]{KVZ23} for the proof of Lemma~\ref{Cor-PPPreprfree}.
Indeed, the proof consists of a series of reformulations of the symmetrized trace: first in terms of $N$ Brownian bridges with time-interval $[0,\beta]$ and a symmetrization, then (using the Markov property, respectively the Chapman--Kolmogorov equations) in terms of a collection of Brownian bridges with various lengths with total sum equal to $N$ and equal initial and terminal sites, and finally  a translation into the language of Poisson point processes. The first two reformulations are due to \cite{G70}, the last one to \cite{ACK10}.

The representation is the starting point  of our analysis. It also gives a frame for the description of the mean-field Bose gas that is explicitly built on an ensemble of loops, which we will be using as order parameters.

\subsection{Our main results: Long loops and ODLRO in the mean-field Bose gas}\label{sec-results}

Let us formulate our precise assumption on the trap potential $w$.
For our purposes, it will be important to control the behaviour of $w$ at its minimum.  We write $\{w=0\}$ for $\{x \in\R^d\colon w(x)=0\}$; similarly for $\{w<\infty\}$ and other sets like that. 

\medskip {\bf Assumption (W).} \textit {We assume that $w\colon\R^d\to[0,\infty]$ is continuous in $\{w<\infty\}$, and there is a parameter $\alpha\in(0,\infty]$ together with a family of functions $W,W_\eps\colon\R^d\to[0,\infty]$,
such that
\begin{itemize}
\item We define
\begin{equation*}
 W_\eps(x)=
 \begin{cases}   
\e^{-\alpha}w(x\e)&\mbox{in the case }\alpha<\infty,\\
\e^{-1} w(x)&\mbox{in the case }\alpha=\infty,
 \end{cases}
 \qquad x\in\R^d, \eps\in(0,1];
\end{equation*}
\item
$\int_{\R^d} \ex^{-\beta \inf_{\eps\in(0,1]}W_\eps(x)}\,\d x<\infty$ for any $\beta\in(0,\infty)$;
\item for any $f\in L^1(\R^d)$, 
\[
\langle W_\eps,f\rangle\rightarrow \langle W,f\rangle\qquad\mbox{as }\eps\downarrow 0;
\]
\item $W$ is continuous in $\{W<\infty\}$ and satisfies $\inf W=0$;
\item $W$ has a unique minimum at $x=0$ if $\alpha<\infty$;
\item $W=0$ in a neighborhood of $x=0$ if $\alpha=\infty$.
\end{itemize}}
In particular, under Assumption (W), $w\geq 0$ and  $w(x)=W_1(x)\geq \inf_{\eps\in(0,1]} W_\eps(x)$ and $\lim_{|x|\to\infty}w(x)=\infty$ so fast that all negative moments $\int_{\R^d}\ex^{-t w(x)}\,\d x$ are finite for $t>0$. 
Hence the $L^2$ operator $-\Delta+w$ 
has a discrete spectrum with spectral gap, i.e. it has an $L^2$-orthonormal basis $(\phi_i^{\ssup w})_{i\in\N}$ of eigenfunctions with associated eigenvalues 
\begin{align}\label{eq:sp_gap}
0<\lambda_1(w)<\lambda_2(w)\leq \lambda_3(w)\leq \dots,
\end{align}
and we take the sign of $\phi_1^{\ssup w}$ such that it is positive everywhere in $\{w<\infty\}$. See \cite[Theorem 4.72, Theorem 4.125]{BHL11} for details.
Same property holds for $-\Delta+W$ and $-\Delta+W_\e$.

We remark that $\alpha=\infty$ implies $W\in\{0,\infty\}$. 
One possible choice is $w=W=\infty\1_{Q^{\rm c}}$ for  a centered box $Q$ in the case $\alpha=\infty$ or the  harmonic trap potential $w(x)=W(x)=|x|^2$ in the case $\alpha=2$. 
Furthermore, in the case $\alpha<\infty$, we have $W(x)=\abs{x}^\alpha W(\frac{x}{\abs{x}})$ for $x\in\R^d$; this includes the case $W(x)=c\abs{x}^\alpha$, in particular the case of a harmonic trap. Certainly, lots of generalisations of Assumption (W) will admit our results, but would require more technical efforts and do not substantially increase the list of interesting potentials. In particular, $\int \ex^{-\beta j w}\d x$ is decreasing in $j\in\N$ and finite.

We fix $\beta\in(0,\infty)$ for the rest of the paper and do not everywhere reflect its dependence in the following. For a realization $\eta$ of the PPP with distribution ${\tt P}_{\beta a,w/a}^{\ssup N}$, we consider the sequence $L(\eta)=(L_i)_{i\in\N}$, defined as the sequence  of all the lengths $\ell(f)$ with $f\in\eta$, ordered according to their size, and counted with multiplicity. That is, $L_i$ is the number of particles in the $i$-th longest loop in the configuration.

Let us recall that the Poisson--Dirichlet distribution with parameters $0$ and $1$ (denoted $\textsf{PD}_1$) is given as the joint distribution of the random variables $(Y_n\prod_{k=1}^{n-1} (1-Y_k))_{n\in\N}$, where $(Y_n)_{n\in\N}$ is an i.i.d.
~sequence uniformly distributed over $[0,1]$.
Note that the sum of the elements of a $\textsf{PD}_1$-distributed sequence is equal to one, i.e., this distribution is in fact a random partition. It is well-known in asymptotics for random permutations, as the joint distribution of the lengths of all the cycles of a uniformly picked random permutation of $1,\dots,N$, ordered according to their sizes and normalized by a factor $1/N$, converges weakly to $\textsf{PD}_1$.

An important quantity is
\begin{equation}\label{rhowdef}
\rho_w=\sum_{j\in\N} \frac{\Wj_j}{(4\pi \beta j)^{d/2}}\in(0,\infty],\qquad  \mbox{where }\Wj_j=\int_{\R^d}\d x \,\ex^{-\beta j w(x) }.
\end{equation}
Furthermore, we need to introduce the {\em pressure}
\begin{equation}\label{pressuredef}
p(u)=\sum_{j=1}^\infty \frac{\ex^{\beta u j}}j  \frac{\Wj_j}{(4\pi \beta j)^{d/2}},\qquad u\in(-\infty,0].
\end{equation}
Then $p$ is analytic in $(-\infty,0)$ and diverges in $(0,\infty)$ with $p'(0)=\beta \rho_w$, where $p'(0)$ is the left-derivative at $0$. For $\chi\in(0, \infty)$, define $u_\chi\in (-\infty,0)$  by $p'(u_\chi)=\chi\beta$ if $\chi< \rho_w$ and $u_\chi=0$ for $\chi\geq \rho_w$. Then we define the sequence
        \begin{equation}\label{alphadef}
 \alpha^{\ssup\chi} =(  \alpha_j^{\ssup\chi})_{j\in\N},\qquad  \alpha_j^{\ssup\chi}=\frac{\ex^{\beta u_\chi j}}\chi \frac{\Wj_j}{(4\pi \beta j)^{d/2}},\qquad j\in\N.
\end{equation} 
Then, in the limit $\chi\downarrow0$; we have $u_\chi\to-\infty$; more precisely, $u_\chi=\frac {1+o(1)}\beta\log( \frac{(4\pi\beta)^{d/2}}{\Wj_1}\chi)$ and hence $\alpha_j^{\ssup\chi}\to\delta_{j,1}$. 
We extend the definition by taking
$\alpha^{\ssup0}=(1,0,0,\dots)$,
$\alpha^{\ssup \infty}=0$.


\begin{theorem}[Asymptotics of reduced density matrix and loop length distribution]\label{thm-ODLROlongloops}
Suppose that the trap potential $w$ satisfies Assumption (W) and that $\rho_w<\infty$. Pick a bounded sequence $(a_N)_{N\in\N}$ in $(0,\infty]$ and recall $\chi=\lim_{N\to\infty} Na_N^{d/2}\in[0,\infty]$ as in \eqref{chidef}. 
    \begin{enumerate}
        \item{\bf Supercritical case: $\boldsymbol{\chi>\rho_w}$.} The following holds in the limit as $N\to\infty$:
            \begin{enumerate}
\item In weak $L^2$-sense,
       \begin{equation}\label{gammaasy}
            \gamma^{\ssup {a_N}}_N(x,y)= N\Big(1- \frac{\rho_w}{\chi}\Big)\phi_1^{\ssup{w/a_N}}(x)\phi_1^{\ssup{w/a_N}}(y)(1+o(1)),\qquad x,y\in\R^d.
        \end{equation}
\item The distribution of the loop lengths $\frac {1}{N(1-\rho_w/\chi)} (L_i)_{i\in\N }$ under ${\tt P}_{\beta a_N,w/a_N}^{\ssup N}(\cdot \mid {\mathfrak N}=N)$ converges to $\textsf{PD}_1$.

\item The distribution of the sequence $\frac 1N(i X_i)_{i\in\N}$ under ${\tt P}_{\beta a_N,w/a_N}^{\ssup N}(\cdot \mid {\mathfrak N}=N)$  converges in product topology on $\ell^\infty$ towards $\alpha^{\ssup\chi}$.

\item The free energy defined in \eqref{freeenergyMF} is identified as
$$
{\rm f}_{\rm MF}(\beta,\chi)=-\frac{p(0)}{\beta \chi}+\lambda_1(W)(1-\smfrac{\rho_w}\chi)\lim_{N\to\infty}a_N.
$$
            
            \end{enumerate}
                 
        \item {\bf Subcritical case: $
            \boldsymbol{\chi<\rho_w}$.} The following holds in the limit as $N\to\infty$:
            \begin{enumerate}
            \item There is a $c\in(0,\infty)$ such that
        \begin{equation}\label{gammabound}
            \gamma^{\ssup {a_N}}_N(x,y)= \Ocal\rk{a_N^{-d/2}\ex^{-c\abs{x-y}a_N^{-1/2}}},\qquad x,y\in\R^d.
        \end{equation}
        \item Under ${\tt P}_{\beta a,w/a}^{\ssup N}(\cdot \mid {\mathfrak N}=N)$, the sequence $\frac 1N (i X_i)_{i\in\N }$ converges weakly in $\ell^1$-norm to $\alpha^{\ssup\chi}$.
        
        \item  The free energy defined in \eqref{freeenergyMF} is identified as
$$
{\rm f}_{\rm MF}(\beta,\chi)=
\begin{cases}
-\frac{p(u_\chi)}{\beta \chi}+u_\chi,&\mbox{if }\chi>0,\\
-\infty&\mbox{if }\chi=0.
\end{cases}
$$
\end{enumerate}
    \end{enumerate}

\end{theorem}

The proofs of Theorem~\ref{thm-ODLROlongloops} (1)(a)--(c) and (2)(a)--(b) are in Section~\ref{sec-ProofODLRO} and \ref{sec-proofthm(2)} respectively, and the identification of the free energy is in Section~\ref{sec-freeenergy}. Our main proof methods are spectral-theoretic (as it concerns the term $\BB_{x,y}{\ssup{\beta  ar}}$ in \eqref{PPPreprfree}), combinatorial (for handling the two probability terms in \eqref{PPPreprfree}); their base is probabilistic, since we will be relying on the useful independence properties of the PPP.

Let us draw consequences about the Bose--Einstein phase transition from Theorem~\ref{thm-ODLROlongloops}. From \eqref{gammaasy} it quickly follows that $\frac 1N \sigma_N^{\ssup{a_N}}\to 1- \rho_w/\chi>0$ (i.e., ODLRO holds), and from \eqref{gammabound} it easily follows that $\sigma_N^{ \ssup{a_N}}\leq O(1)$ (i.e., ODLRO does not hold):

\begin{corollary}[Consequences for (non-)occurrence of ODLRO]
    Equation \eqref{gammaasy} implies ODLRO while Equation \eqref{gammabound} implies its absence. 
\end{corollary}
\begin{proof}
 (1)   In \eqref{Gammaeigenvalue} we use $\phi_1^{\ssup{w/a}}$ for $f$ and obtain
    \begin{eqnarray}
        \sigma_N^{\ssup a}\ge \langle \phi_1^{\ssup{w/a}}, \gamma^{\ssup {a_N}}_N\phi_1^{\ssup{w/a}}\rangle\ge  N\rk{1-\frac{\rho_w}\chi}\norm{\phi_1^{\ssup{w/a}}}_{2}^2= N\rk{1-\frac{\rho_w}{\chi}}\, .
    \end{eqnarray}

  (2)  To prove absence of ODLRO, we use Young's convolution inequality to estimate
    \begin{eqnarray}
        \sigma_N^{\ssup a}=\sup_{f\in L^2(\R^d)\colon \|f\|_2=1}\langle f, \gamma^{\ssup a}_N f\rangle
        \le \Ocal(a_N^{-d/2}) \norm{\ex^{-c\abs{\cdot }a_N^{-1/2}}}_1=\Ocal\rk{a_N^{-d/2} a_N^{d/2}}=\Ocal(1)\, .
    \end{eqnarray}
\end{proof}

\begin{remark}[Total mass in micro- and macroscopic loops] 
Theorem~\ref{thm-ODLROlongloops} implies that the total mass of particles in microscopically long loops is $\sim N [\min\{1,\rho_w/\chi\}]$, while the total mass in macroscopically long loops is $\sim N(1-\rho_w/ \chi)_+$. This shows a phase transition in $\chi$ at $\chi=\rho_w$ between occurrence and non-occurrence  of particles in macroscopic loops. This is the famous Bose--Einstein condensation phase transition. When putting $\chi=1$ (i.e., in the semi-classical regime), it can be found at $\beta=\beta_{\rm c}$, defined by $\rho_w(\beta_{\rm c})=1$. Note that Theorem~\ref{thm-ODLROlongloops} implies that there is only $o(N)$ particles in  other loops, i.e., in mesoscopically long loops.

Furthermore, for $\chi=0$ (i.e., $a_N=o(N^{-2/d})$), we observe that only loops of length one contribute, and the free energy is equal to $-\infty$. For $\chi=\infty$  (i.e., $a_N\gg N^{-2/d}$), we observe hundred percent condensation, more precisely, hundred percent of particles are in macroscopic loops. This includes the special case $a_N=1$, where  \cite{AK08} identified the free energy with other methods, but had no assertion about loop lengths.
\hfill$\Diamond$
\end{remark}

\begin{remark}[Spatial distribution of the condensate]\label{rem-spatialdistcond} The spatial density of the location of the condensate is equal to $x\mapsto \frac 1N \gamma^{\ssup{a_N}}(x,x) $. Note from Lemma~\ref{Cor-PPPreprfree} that this is the density of the location of the initial site of the  sample loop in the loop soup, weighted  with the number of particles in that loop (since the weight of a length-$r$ loop starting from $x$ is equal to $\frac 1r \BB_{x,x}^{\ssup{\beta a r, w/a}}$). This clarifies the suggestion by Feynman \cite{F53} about the loop weights as an order parameter for the condensate in the loop soup. According to Theorem~\ref{thm-ODLROlongloops}(1)(a), the spatial condensate density is asymptotically distributed with density equal to  the total condensate mass times $(\phi_1^{\ssup {w/a}})^2$. According to Lemma ~\ref{lem_spectralscaling} below, this density has a spatial rescaling with scaling parameter $a^{-\alpha/(\alpha+2)}$ and rescaled shape equal to $(\phi_1^{\ssup W})^2$. Hence, the condensate shrinks together to the origin in the case $\alpha>0$ and is distributed like the square of the principal eigenfunction of $-\Delta+\infty\1_{\{W=\infty\}}$ in the case $\alpha=0$.
\hfill$\Diamond$
\end{remark}

\begin{remark}[Finiteness of $\rho_w$]
Under Assumption (W), $(\int \ex^{-\beta j w}\d x)_{j\in\N}$ is bounded, and hence $\rho_w$ is finite at least in $d\geq 3$.

In the special case that $w=\infty\times\1_{Q^{\rm c}}$, where $Q$ is the centred box of volume $1/\rho$, then $\int \ex^{-\beta j w}\d x=1/\rho$ and hence $\rho_w= \frac 1\rho(4\pi\beta)^{-d/2}\zeta(d/2)$, where $\zeta$ denotes the Riemann zeta function. Here, $\rho_w$ is finite only in dimension $d\geq 3$, and $ \lambda_1(w)$ is equal to the Dirichlet zero eigenvalue of the Laplace operator in $Q$ with corresponding principal eigenfunction $\phi_1$. This is -- up to scaling -- equal to the situation in the free Bose gas in the thermodynamic regime with Dirichlet boundary condition, see \cite{KVZ23}.

However, in the case of a harmonic trap, or, more generally in  the case that $w(x)\sim D |x|^2$ for $x\to0$ for some $D>0$, then $\int \ex^{-\beta j w}\d x\sim (\pi/ \beta j D) ^{d/2}$, as one sees by a standard Gaussian approximation. In this case, $\rho_w$ is finite also in $d=2$. It is no problem to construct examples of potentials $w$ such that $\rho_w<\infty$ also in $d=1$.
\hfill$\Diamond$
\end{remark}

\subsection{Literature remarks}\label{sec-literature}

The study of quantum gases, in particular the Bose gas and its statistical mechanics and condensation, is a huge fascinating subject that provides many challenging questions and involves a lot of mathematical ansatzes and toolboxes, see \cite{PS01,PS03} for extensive summaries.

Interacting quantum gases in various mean-field approximations were recently studied in a series of papers; see the extensive summary \cite{FKSS20}. It contains a wealth of ansatzes and formulas, references and summaries of recent results, mostly by the authors. The small-$a$ regime (in our notation) is coupled in \cite{FKSS20}  with other rescalings (e.g., of the interaction strength), but is also considered for fixed number of particles in a fixed box. Throughout this series of papers, the gas is assumed to be confined to a box with periodic boundary condition, and it is considered in the grand-canonical setting. The main ansatz, like in many investigations in the mathematical physics community, is via the formalism of the second quantisation, i.e., in terms of a formulation using annihilation and creation operators on the Fock space. In that series of works, also the description in terms of Brownian bridges (called a path-integral approach there, as usual in the mathematical physics community) is derived in a way that is alternative to the way that is chosen here (we rely on Ginibre's Feynman--Kac formula via the density of the operator $\ex^{\beta \Delta}$, the Brownian bridge measure), via a number of presentations. This formula is used in \cite{FKSS20} for deriving the $a\downarrow 0$ limit for fixed particle number and fixed box; indeed, the partition sum converges towards the one of an interacting classical gas of $N$ particles. The regime that we consider in the present paper was not considered in \cite{FKSS20}.

The {\em semiclassical limit} (i.e., the choice $a_N\sim N^{-2/d}$) at positive temperature with an interaction scaled by $\frac 1N$, recently attracted some interest, both for fermions \cite{LMT19, FLS18} and for bosons \cite{DS21, BK27}. \cite{DS21} studied a special case of the regime that we investigate in this paper (however, with interaction!),  where $a_N\sim N^{-2/d}$ and $d=3$ and the harmonic trap  $w(x)=\omega |x|^2$ for some $\omega>0$ and a pair-interaction potential $v$ satisfying an upper bound of its Hessian matrix in terms of $\omega$. They managed to prove, among other things, the existence of a phase transition in $\beta$ at some critical value $\in(0,\infty)$: above that value, ODLRO holds and that the condensate concentrates asymptotically in one singe point, the origin, and below that value, BEC does not occur. Their methods are very functional-analytic, start with the Fock-space formulation and rely on reformulations in the Fourier world.

It is the goal of the present  paper to re-prove and re-interpret such results on one hand in greater generality with respect to the regimes of $(a_N)_{N\in\N}$ and the shape of the trap potential $w$, and on the other hand to give probabilistic proofs that show the benefits of the Feynman--Kac representation by Ginibre and the Poisson point process representation introduced in \cite{ACK10} and turn the attention to the Brownian loop soup as an object of its own interest.  (The goal of \cite{BK27} is to do this also with interactions in greater generality than \cite{DS21}.) Rigorous considerations of  Brownian bridges as an order parameter for Bose gases appeared in a few works yet, starting with phenomenological discussions in \cite{U06} and discussions of the relation between long loops and condensate in \cite{S93,S02}. More recently in \cite{FKSS20} conceived the rescaled interaction of the Brownian loops in $d=4$ as a regularization as the intersection local time as a possible ansatz for deriving $\phi^4$-theories. Furthermore, in \cite{BKM24} interactions only within the same loop were admitted in the gas and a related kind of condensation phase transition was proved in connection with the famous self-avoiding walk problem. Finally, in our recent paper \cite{KVZ23}, where ODLRO was explicitly proved via this route for the free Bose gas, a contribution that was apparently missing yet. In the case $a_N=1$, for $w=\infty$ outside a box $\Lambda\subset\R^d$ and continuous inside $\Lambda$, in \cite[Theorem 1.6]{AK08} it was shown that 
$$
\lim_{N\to\infty}\frac 1N\log Z_N(\beta,1, w)=\beta \lambda_1(w).
$$
The proof also starts from the well-known trace formula involving Brownian cycles, but uses a somewhat sophisticated combinatorial approach, which appear unfeasible in the case $a_N\downarrow 0$. It is not difficult to include  interactions (of course, with a prefactor of $\frac 1N$). Using a comparison to  the analogous model with one long Brownian path instead of an ensemble of many cycles, this result was interpreted in \cite{AK08} as the fact that the Bose gas behaves as if it would consist only of one long cycle. But there was no deeper understanding provided in \cite{AK08}.




\section{Preparations for the proofs}\label{sec-preparations}

In this section, we prepare for all the forthcoming proofs by the following: In Section~\ref{sec-funcana} we provide upper bounds and asymptotics for the intensity measure and its total mass of the Poisson point process (PPP), and clarify some spectral scaling properties. In Section~\ref{sec-conc} we show that the particle number in the PPP is with high probability close to its expectation; and we show the same for the number of particles in small loops. In Section~\ref{sec-longloops} we give precise asymptotics for the distribution of the number of particles in long loops, which leads in Section~\ref{sec-lowbound} to a precise lower bound for the distribution of the total number of particles (the denominator in \eqref{PPPreprfree}).

\subsection{Functional analytic properties}\label{sec-funcana}

In this section we provide bounds and precise asymptotics for the intensity measure and its total mass of the crucial Poisson point process that we introduced in  Section~\ref{sec-PPPintro}. We keep $\beta\in(0,\infty)$ fixed and are under Assumption (W) for the potential $w$. 
Recall from \eqref{eq:sp_gap} that the operator $-\Delta+w$ has eigenvalues $0<\lambda_1(w)<\lambda_2(w)\leq \lambda_3(w)\leq \dots$  with a corresponding $L^2$-orthonormal system $(\phi_i^{\ssup w})_{i\in\N}$ of eigenfunctions such that $\phi_1^{\ssup w}$ is positive whenever $w$ is finite.

Recall the Brownian bridge measure $ \BB^{\ssup \beta}_{x,y}$ defined in \eqref{nnBBM}, which is a regular Borel measure on $\Ccal_\beta$ with total mass equal to the Gaussian density,
\begin{equation}\label{Gaussian}
\BB_{x,y}^{\ssup \beta}(\Ccal_{\beta})=g_{\beta}(x,y)=\frac {\P_x(B_\beta\in\d y)}{\d y}=(4\pi\beta)^{-d/2}{\rm e}^{-\frac 1{4\beta}|x-y|^2}.
\end{equation}
We refer the reader to Appendix ${\rm A}$ of \cite{Sz98} for more details on Brownian bridge measures. Recall from \eqref{Vdef} its integrated and weighted version on loops and introduce its total mass
\begin{equation}\label{gammadef}
\oldgamma_{\beta,w}
=\BB^{\ssup{\beta,w}}(\Ccal_\beta)
=\int_{\R^d}\d x\int_{\Ccal_\beta}\BB_{x,x}^{\ssup{\beta}}(\d f)\,\ex^{-\int_0^\beta w(f(s))\,\d s}.
\end{equation}
This is finite under Assumption (W), see for example \cite{BHL11}.
We write $\mu(f)$ for the integral of a function $f$ with respect to a measure $\mu$. Recall that we write ${\tt P}_{\beta,w}^{\ssup N}$ for the probability measure of a Poisson point process (PPP) $\eta=\sum_f \delta_f$  with intensity measure $\nu_{\beta, w}^{\ssup N}$ defined in \eqref{nudef}. This intensity measure has total mass 
\begin{equation}\label{totalmassnu}
\nu_{\beta, w}^{\ssup N}(\widehat \Ccal_\beta)=  \sum_{k=1}^N\frac 1k \,\oldgamma_{\beta k,w}.
\end{equation}

A standard eigenvalue expansion (see for example \cite[Theorem 4.72]{BHL11}) gives that
\begin{equation}\label{eigenvalueexp_with_xy}
\BB_{x,y}^{\ssup{\beta,w}}(\Ccal_{\beta})=
\E_x\big[\ex^{-\int_0^\beta w(B_s)\,\d s}\1\{B_\beta\in\d y\}\big]\big/\d y=\sum_{i\in\N} \ex^{-\beta\lambda_i(w)}\phi_i^{\ssup w}(x)\phi_i^{\ssup w}(y),\qquad  x,y\in\R^d,
\end{equation}
and
\begin{equation}
\oldgamma_{\beta,w}=\int_{\R^d} \BB_{x,x}^{\ssup{\beta,w}}(\Ccal_{\beta})\,\d x=\sum_{i\in\N}\ex^{-\beta \lambda_i(w)}.
\end{equation}

Driven by \eqref{Zdef}, now we 
replace $\beta$ by $\beta a$ and $w$ by $w/a$. 
We need to know, as $a\downarrow 0$, the asymptotics of $\BB_{x,y}^{\ssup{\beta a,w/a}}(\Ccal_{\beta})$ and of 
\begin{equation}\label{tjadef}
\begin{aligned}
t_{j,a}&=\oldgamma_{\beta a j,w/a}
={\tt E }_{\beta a, w/a}[\#\{f\in\eta\colon \ell(f)=j\}]
=\int_{\R^d} \BB_{x,x}^{\ssup{\beta a  j,w/a}}(\Ccal_{\beta a  j})\,\d x \\
   &=\int_{\R^d} \frac{\E_x\big[\ex^{-\frac 1a\int_0^{\beta a j}w(B_s)\,\d s}\1\{B_{\beta a j}\in\d x\}\big]}{\d x}\,\d x
=\sum_{i\in\N}\ex^{-\beta a j \lambda_i(w/a)}.
\end{aligned}
\end{equation}

We first state rescaling properties of the spectrum of $-\Delta+w/a$:

\begin{lemma}[Spectrum of $-\Delta+\smfrac wa$]\label{lem_spectralscaling}Assume that $w$ satisfies Assumption (W) with $\alpha<\infty$.
Then, as $a\downarrow0$, with $\e=a^{1/(\alpha+2)}$, 
\begin{equation}\label{eq:2.7}
    a\lambda_{i}(\smfrac wa)\sim a^{\alpha/(2+\alpha)}\lambda_i(W)
    \text{ for $i\in\{1,2\}$, and }
    \phi_1^{\ssup {w/a}}(x)\sim\phi_1^{\ssup W}(x\e^{-1})\e^{-d/2}\mbox{ in }L^2\mbox{-sense}.
\end{equation}
In particular, the spectral gap satisfies
\begin{equation}\label{spectralgapasy}
a\big[\lambda_2(\smfrac wa)-\lambda_1(\smfrac wa)\big]\sim a^{\alpha/(2+\alpha)}\big[\lambda_2(W)-\lambda_1(W)\big],
\end{equation}
and the last bracket is positive. 
\end{lemma}

\begin{proof} Recall $W_\e(x)=\e^{-\alpha} w(x\e)$ from Assumption (W) for $\e\in(0,1]$. Then the spectra of $-\Delta+\frac wa$ and $-\Delta +W_\e$ with $\e=a^{1/(\alpha+2)}$ stand in a one-to-one correspondence with each other. Indeed, we easily see that, for any $i\in\N$, the $i$-th eigenvalue/eigenfunction pairs $(\lambda_i(w/a),h_i)$ and $(\lambda_i(W_{\e}),g_{\e,i})$ satisfy
\begin{equation}\label{spectrumrescaling}
\e^2\lambda_i(\smfrac w a)= \lambda_i(W_{\e})\qquad\mbox{and}\qquad g_{\e,i}(x)=\e^{d/2} h_i(x\e),\quad x\in\R^d.
\end{equation} 

Now we show that $\lim_{\e\downarrow 0 }\lambda_1(W_\e)=\lambda_1(W)$, which implies the first statement in \eqref{eq:2.7}. We use the Rayleigh--Ritz principle, $\lambda_1(W)=\inf_{g\in L^2(\R^d)\colon\|g\|_2=1}\langle (-\Delta+W)g,g\rangle$. Taking $g$ as the normalized principal eigenfunction of $-\Delta+W$, we get, in the limit $\e\downarrow0$,
\begin{equation}\label{eq:lambda_1}
\lambda_1(W_\e)\le\langle(-\Delta+W_\e)g,g\rangle=-\langle\Delta g,g\rangle+\langle W_\e g,g\rangle\rightarrow
-\langle\Delta g,g\rangle+\langle W g,g\rangle=\lambda_1(W).
\end{equation}

For the other direction, let
$g_\e\in L^2(\R^d)$ be the normalized eigenfunction corresponding to $\lambda_1(W_\e)$. Then for every $\e$,
\begin{align}\label{eq:211}
\lambda_1(W_\e)
=\langle(-\Delta+W_\e)g_\e,g_\e\rangle=\|\nabla g_\e\|_2^2+\langle W_\e g_\e,g_\e\rangle.
\end{align}
Since $g_\e,W_\e\ge 0$, by \eqref{eq:lambda_1}, $(\|\nabla g_{\e}\|_2)_{\e\in(0,1]}$ is bounded. Pick a sequence $(\e_n)_{n\in\N}$ with $\lim_{n \to\infty}\e_n=0$. We deduce from \cite[Theorem 2.18, Theorem 8.6]{LL01} that there is some $g_0\in L^2(\R^d)$ such that, along some subsequence that we still denote $(\e_n)_n$,
\begin{align}
\nabla g_{\e_n}\rightarrow \nabla g_0,  \qquad g_{\e_n}\rightarrow g_0, \text{ weakly in }L^2; \qquad g_{\e_n}\1_{[-R,R]^d}\rightarrow g_0\1_{[-R,R]^d}\text{ strongly in }L^2 \mbox{ for any }R>0.
\end{align}
Since $g_{\e_n}$ is continuous for any $n$, we furthermore have that $g_{\e_n}$ converges almost everywhere to $g_0$, \cite[Corollary 8.7]{LL01}.
By Fatou's lemma and \cite[Theorem 2.11]{LL01} (lower semi-continuity of $g\mapsto\|\nabla g\|_2^2$),
\begin{align}\label{eq:lambda_3}
\liminf_{n\rightarrow\infty}\lambda_1(W_{\e_n})=\liminf_{n\rightarrow\infty}(\|\nabla g_{\e_n}\|_2^2+\langle W_{\e_n}g_{\e_n},g_{\e_n}\rangle)
\ge\|\nabla g_0\|_2^2+\langle Wg_0,g_0\rangle.
\end{align}

Moreover, 
\begin{align}\label{eq:g_norm}
\|g_0\|_2=\lim_{R\rightarrow\infty}\|g_0\1_{[-R,R]^d}\|_2=\lim_{R\rightarrow\infty}\lim_{n\rightarrow\infty}\|g_{\e_n}\1_{[-R,R]^d}\|_2
\ge1-\lim_{R\rightarrow\infty}\lim_{n\rightarrow\infty}\|g_{\e_n}\1_{\R^d\setminus[-R,R]^d}\|_2.
\end{align}
As $R\rightarrow\infty$, by \eqref{eq:lambda_1} and \eqref{eq:211},
\begin{align*}
\|g_{\e_n}\1_{\R^d\setminus[-R,R]^d}\|_2^2
\le\frac{1}{\inf_{\R^d\setminus[-R,R]^d}W_{\e_n}}\int W_{\e_n}g_{\e_n}^2\d x
\le\frac{\lambda_1(W)}{\inf_{\R^d\setminus[-R,R]^d}W_{\e_n}}
\rightarrow \frac{\lambda_1(W)}{\inf_{\R^d\setminus[-R,R]^d}W}.
\end{align*}
By Assumption (W), we know $W(R x)=R^\alpha W(x)$ and $\inf_{\R^d\setminus[-1,1]^d}W>0$,
so
\begin{align*}
\lim_{R\rightarrow\infty}\lim_{n\rightarrow\infty}\|g_{\e_n}\1_{\R^d\setminus[-R,R]^d}\|_2^2
\le \lim_{R\rightarrow\infty}\frac{\lambda_1(W)}{\inf_{\R^d\setminus[-R,R]^d}W}=0.
\end{align*}
Put this into \eqref{eq:g_norm}, we deduce that
\[
\|g_0\|_2\ge 1.
\]

Hence, the right-hand side of \eqref{eq:lambda_3} is not smaller than $\lambda_1(W)$. Together with \eqref{eq:lambda_1}, this implies that $\lim_{\e\downarrow0}\lambda_1(W_\e)=\lambda_1(W)$, as announced.

A slight extension of the above proof also shows that $\e^{d/2} \phi_1^{\ssup{w/a}}(x \e)$ converges uniformly on compacts towards $\phi_1^{\ssup W}$ as $\e\to 0$. In the same way, we can also show that the same is true for the second eigenvalue, based on the Rayleigh--Ritz formula $\lambda_2(W)=\inf_{g\in L^2(\R^d)\colon\|g\|_2=1, g\perp \phi_1}\langle (-\Delta+W)g,g\rangle$, where we denote the principal eigenfunction of $-\Delta+W$ by $\phi_1$. 

Finally, by \cite[Theorem 4.72, Theorem 4.125]{BHL11}, we have $\lambda_2(W)>\lambda_1(W)>0$.
\end{proof}

\begin{lemma}[Asymptotics of $t_{j,a}$]\label{lem-tasy}Assume that $w$ satisfies Assumption (W). For $\alpha=\infty$, read $\alpha/(\alpha+2)$ as $1$.
\begin{enumerate}
\item There is a $C\in(0,\infty)$ such that, for any $j\in\N$ and $a\in(0,\infty)$,
 \begin{equation}
 t_{j,a}\leq (4\pi\beta a j)^{-d/2}\int_{\R^d} \ex^{-\beta j w(x)}\d x\le C a^{-d/2}j^{-d/2-d/\alpha}\, .
 \end{equation}
\item If $a\in(0,1]$ and $j\in\N$, in the limit as $ja^{\alpha/(2+\alpha)}\to0$ (and $j\to\infty$ for the second expression),
\begin{equation}
t_{j,a}\sim (4\pi\beta a j)^{-d/2}\int_{\R^d} \ex^{-\beta j w(x)}\d x\sim (4\pi\beta a j^{1+2/\alpha})^{-d/2}\int_{\R^d} \ex^{-\beta W(x)}\,\d x\, .
\end{equation}

\item There is a $c\in(0,\infty)$ such that, as $ja^{\alpha/(2+\alpha)}\to\infty$, for any two test functions $f,g\in L^2(\R^d)$, (possibly depending on $a$, but with bounded norms),
\begin{equation}\label{xiasy}
\begin{aligned}
\int_{\R^d}\d x\int_{\R^d}\d y\, f(x) g(y)\BB_{x,y}^{\ssup{\beta j a ,w/a}}(\Ccal_{\beta j a})
&=\ex^{-\beta j a\lambda_1(w/a) }\langle f, \phi_1^{\ssup {w/a}}\rangle\,\langle g,\phi_1^{\ssup {w/a}}\rangle\,\big(1+\eps(ja^{\alpha/(2+\alpha)})\big),
\end{aligned}
\end{equation}
provided that $\langle f, \phi_1^{\ssup {w/a}}\rangle\,\langle g,\phi_1^{\ssup {w/a}}\rangle\not=0$, where the error term satisfies $\eps(k)=\Ocal(k^{-\frac d2\,\frac {\alpha+2}\alpha})\ex^{-ck}$ as $k\to\infty$. In particular,
\begin{equation}\label{tasy_harmonic}
t_{j,a}=\ex^{-\beta j a\lambda_1(w/a) }\,\big(1+\eps(ja^{\alpha/(2+\alpha)})\big)\, .
\end{equation}
\end{enumerate}
\end{lemma}

\begin{proof}
The case $\alpha=\infty$ and $a\downarrow 0$ is basically identical with the situation in \cite[Lemma 2.2]{KVZ23}, we the case of $W=\infty\1_{(LU)^{\rm c}}$ is handled with various kinds of boundary conditions (including Dirichlet zero conditions), with $U$ a centred box and $L\in(0,\infty)$ tending to $\infty$. An extension from a box $U$ to the set $\{W=0\}$ under Assumption (W) is clearly no problem. Use the Brownian scaling property to see that the limit as $a\downarrow0$ with fixed $W$ (instead of $w/a$; recall that $W$ takes only values in $\{0,\infty\}$) is equivalent to this limit as $L\to\infty$. The replacement of $W$ by $W_\eps$ with $W_\e\to W$ as $\e\downarrow 0$ for $a\downarrow 0$  is only a minor technical point.  The case where $\alpha=\infty$ and $a\in(0,\infty)$ is fixed is even easier to prove; we leave the details to the reader.

Hence, we assume that $\alpha\in(0,\infty)$.

(1) Conditional on the Brownian motion $B$, apply Jensen's inequality for the probability measure $ \frac 1{\beta a j}\int_0^{\beta a j}\d s$ and the negative-exponential map, we get
\begin{equation}\label{Proof1}
\begin{aligned}
t_{j,a}&=(4\pi\beta a j)^{-d/2}\int_{\R^d}\d x\,\E_x\Big(\ex^{-\beta j\frac 1{\beta a j}\int_0^{\beta a j}w(B_s)\,\d s}\Big|B_{\beta a j}=x\Big)\\
&\leq (4\pi\beta a j)^{-d/2}\int_{\R^d}\d x\,\frac 1{\beta a j}\int_0^{\beta a j}\d s\,
\E_x\Big(\ex^{-\beta j w(B_s)}\Big|B_{\beta a j}=x\Big)\\
&=\int_{\R^d}\d y\,\ex^{-\beta j w(y)}\frac 1{\beta a j}\int_0^{\beta a j}\d s\, \int_{\R^d}\d x\,p_s(x-y)p_{\beta a j-s}(y- x)\\
&=\int_{\R^d}\d y\,\ex^{-\beta j w(y)}\frac 1{\beta a j}\int_0^{\beta a j}\d s\,p_{\beta a j}(0)=(4\pi\beta a  j)^{-d/2} \int \ex^{-\beta j w(x)}\d x,
\end{aligned}
\end{equation}
where we used the Gaussian density $p_t$ with variance $2t$ and used their convolution property. Furthermore, recalling $W_\eps(x)=\eps^{-\alpha} w(x\eps)$, we see, making a change of variables $y=x j^{-1/\alpha}$, that
\begin{equation}\label{boundnegexpofw}
\int\ex^{-\beta j w}\d x=j^{-d/\alpha}\int_{\R^d} \ex^{-\beta j w(xj^{-1/\alpha})}\,\d x=j^{-d/\alpha}\int \ex^{-\beta  W_{j^{-1/\alpha}}}\d x
\leq j^{-d/\alpha}\int \ex^{-\beta \inf_{\eps\in(0,1]}W_\eps}\d x,
\end{equation}
which is finite, according to Assumption (W).


(2) The upper bound follows from (1), in particular \eqref{boundnegexpofw}, which makes it possibly to carry out the limit as $j\to\infty$ under the integral, by the virtue of the bounded convergence theorem.

We turn to the proof of the lower bound. We write $\overline\BB_{x,y}^{\ssup\beta}=\BB_{x,y}^{\ssup\beta}/\BB_{x,y}^{\ssup\beta}(\Ccal_\beta)$ for the normalized version of the Brownian bridge measure. Pick a large $M$, then, by Jensen's inequality, Brownian scaling, and a change of variables $y= x j^{1/\alpha}$ and $r=s\beta a j$,
\begin{align*}
      (4\pi\beta a j)^{d/2}t_{j,a}
      &\ge \int_{\abs{x}<Mj^{-1/\alpha}}\ex^{-\beta j w(x)}\exp\rk{-\frac{1}{a}\int_0^{\beta a j}\overline\BB_{0,0}^{\ssup{\beta a j}}\ek{w(x+B_s)-w(x)}\,\d s}\,\d x\\
      &= \int_{\abs{x}<Mj^{-1/\alpha}}\ex^{-\beta j w(x)}
      \exp\Big(-\beta \int_0^1\d r\, \overline\BB_{(0,0)}^{\ssup 1}\big(W_{j^{-1/\alpha}}(y+j^{1/\alpha}B_r\sqrt{\beta a j})-W_{j^{-1/\alpha}}(y)\big)\Big).
\end{align*}
Observe that 
the $B_r$-depending term in the argument of $W_{j^{-1/\alpha}}$ vanishes, since $j^{1/\alpha}\sqrt{aj}=(a^{\alpha/(\alpha+2)} j)^{(\alpha+2)/2\alpha}$, which vanishes, according to our assumption. Together with the fact that $(W_\eps)_\e$ converges on compact sets, the integrand in the $r$-integral vanishes in this limit, uniformly in $y$ on the integration area, and we have the first result by taking $M\rightarrow\infty$.
Finally, note that 
\[
      \int\ex^{-\beta j w(x)}\d x
      = j^{-d/\alpha}\int \ex^{-\beta W_{j^{-1/\alpha}}(y)}\d y,
\]
we have the second conclusion by taking $j\rightarrow\infty$.

(3) 
We rely on the eigenvalue expansion in \eqref{eigenvalueexp_with_xy} and use the small-$a$ asymptotics of spectral gap from Lemma~\ref{lem_spectralscaling}, which allows us to replace the entire sum by the first summand only. 

By Jensen's inequality, we find that for every $t>0$,
\begin{equation}
\begin{aligned}
     \sum_{i\in\N }\ex^{-t \lambda_i(w)}
     &\leq \frac 1t\int_0^t\d s\, \int_{\R^d}\d x\, \E_x\big[\ex^{-\beta w(B_s)} \1\{B_t\in\d x\}\big]/\d x \\
    & =\frac 1t\int_0^t\d s\, \int_{\R^d}\d x\,\int_{\R^d}\d y\,g_s(x,y)\ex^{-\beta w(y)}g_{t-s}(y,x)(4\pi\beta t)^{-d/2}
     \le C t^{-d/2}\int_{\R^d}\ex^{-t w(x)}\d x\, ,
    \end{aligned}
\end{equation}
where we used the convolution property of the Gaussian kernel $g_s(x,y)$ with variance $2s$.
This implies that for any $a\in(0,1]$, $j\in\N$ and $\delta\in(0,1)$,
\begin{equation}
\begin{aligned}
    \sum_{i\ge 2}\ex^{-\beta a j\lambda_i(w/a)}
    &\le \ex^{-\beta a j\l_2(w/a)(1-\delta)}\sum_{i\ge 2}\ex^{-\beta aj\delta\l_i(w/a)}\\
    &\le \ex^{-\beta a j\l_2(w/a)(1-\delta)} C(aj\delta)^{-d/2}\int\ex^{-\beta  j \delta w}\\
    &\leq\ex^{-\beta aj\l_2(w/a)(1-\delta)} C_\delta \big(j\,a^{\alpha/(\alpha+2)}\big)^{-\frac d2\,\frac {\alpha+2}\alpha}\, ,
     \end{aligned}
\end{equation}
where $C_\delta$ depends only on $\beta$ and $\delta$, and  the last step used also the second assertion in (1). Hence,
$$
\begin{aligned}
t_{j,a}&=\sum_{i\in \N}\ex^{-{
    \beta aj}\lambda_i(w/a)}
    \leq \ex^{-{
    \beta aj}\lambda_1(w/a)}\Big(1+\ex^{{
    \beta aj}\lambda_1(w/a)}\ex^{-\beta aj\l_2(w/a)(1-\delta)} C_\delta \big(j\,a^{\alpha/(\alpha+2)}\big)^{-\frac d2\,\frac {\alpha+2}\alpha}\Big).
    \end{aligned}
    $$
    Now, by \eqref{spectralgapasy}, we can pick $\delta$ so small that, for some $c>0$, the product of the two exponentials is not larger than $\ex^{- c a^{\alpha/(\alpha+2)}}$ in the limit that we consider. This implies \eqref{tasy_harmonic}.

The proof of \eqref{xiasy} is based on the preceding and on the Cauchy--Schwarz inequality and Parseval's identity as follows:
$$
\sum_{i\geq 2}\ex^{-\beta a j\delta \lambda_i(w/a)}\big|\langle f, \phi_i^{\ssup {w/a}}\rangle\,\langle g,\phi_i^{\ssup {w/a}}\rangle\big|
\leq \Big(\sum_{i\in\N}\langle f, \phi_i^{\ssup {w/a}}\rangle^2\Big)^{1/2}\Big(\sum_{i\in\N}\langle g, \phi_i^{\ssup {w/a}}\rangle^2\Big)^{1/2}
=\|f\|_2\,\|g\|_2.
$$
This bound is also sufficient if $f$ or $g$ depend on $a$, but have norms that are bounded in $a$.
\end{proof}

We can immediately draw a conclusion for the expected numbers of number of particles in the PPP. Recall that ${\mathfrak N}(\eta)$ is the number of particles in a PPP $\eta$, and  the number $\rho_w=\sum_{j\in\N}(4\pi\beta j)^{-d/2}\int \ex^{-\beta j w}\d x$. Now, in the case that $\lim_{N\to\infty}a_N=0$, we introduce the threshold 
\begin{equation}\label{TNdef}
T_N=\big\lfloor a_N^{-\alpha/(\alpha+2)}\big(\log\smfrac 1{a_N}\big)^{1/2}\big\rfloor,\qquad N\in\N,
\end{equation}
while we put $T_N=\lfloor (\log N)^{1/2}\rfloor$ in the case that $(a_N)_{N\in\N}$ is bounded, but does not vanish. Note that $1\ll T_N\leq N^{\frac 2d\,\frac\alpha{\alpha+2}+o(1)}$. Then
 ${\mathfrak N}^{\ssup{\rm short}}(\eta)=\sum_{k\leq T_N} \sum_{f\in\eta\colon \ell(f)=k}\ell(f)$ denotes the number of particles in loops of lengths $\leq T_N$ in the PPP $\eta$, which we call the  {\em short loops}. The other loops are called {\em long}, and ${\mathfrak N}^{\ssup{\rm long}}(\eta)=\sum_{k=1+ T_N}^N \sum_{f\in\eta\colon \ell(f)=k}\ell(f)$ is the number of particles in long loops.

\begin{corollary}\label{cor:N_short}
As $N\to\infty$, ${\tt E}_{\beta a_N,w/a_N}^{\ssup N}({\mathfrak N})\sim \rho_w a_N^{-d/2}$ and ${\tt E}_{\beta a_N,w/a_N}^{\ssup N}({\mathfrak N}^{\ssup{\rm short}})\sim \rho_w a_N^{-d/2}$.
\end{corollary}

\begin{proof}
Note that  ${\tt E}_{\beta a,w/a}^{\ssup N}({\mathfrak N})=\sum_{j=1}^N t_{j,a}$. Now the lower bound is shown by restricting the sum to $j\leq M$ for some $M\in\N$, using the asymptotics of Lemma~\ref{lem-tasy}(2) and making $M \to\infty$ afterwards. The upper bound directly follows from Lemma~\ref{lem-tasy}(1). The same applies when cutting the $k$-sum at $T_N$, since $T_N\to\infty$. 
\end{proof}

\subsection{Concentration inequalities}\label{sec-conc}

Next, we prove a concentration inequality for the number of particles and for the number of particles in short loops in the configuration of the PPP. We write $[x]_+$ for the positive part of $x\in\R$.

\begin{proposition}\label{prop:2.6}
Assume that $(a_N)_N$ is a bounded sequence in $(0,\infty)$ and that $w$ satisfies Assumption (W). Then for any  $k\in(0,\infty)$ (possibly depending on $N$), in the limit as $N\to\infty$, the following holds.
\begin{enumerate}
    \item If $\kappa<\beta \l_1(W)$, then 
    \begin{equation}
       \log {\tt P}_{\beta a_N,w/a_N}^{\ssup N}\big(\big|{\mathfrak N}-\Ett_{\beta a_N,w/a_N}^{\ssup N}[{\mathfrak N}]\big|>k\big)\le - \kappa ka_N^{\frac\alpha{\alpha+2}}+a_N^{-[\frac d2-\frac{2\alpha}{\alpha+2}]_++o(1)}.
    \end{equation}
    \item For any $\kappa>0$,
    \begin{equation}\label{shortloopsconc}
       \log {\tt P}_{\beta a_N,w/a_N}^{\ssup N}\big(\big|\Ns-\Ett_{\beta a_N,w/a_N}^{\ssup N}\ek{\Ns}\big|> k\big)
        \le -\kappa k a_N^{\frac\alpha{\alpha+2}}+a_N^{-[\frac d2-\frac{2\alpha}{\alpha+2}]_++o(1)}
         \end{equation}
    In particular, pick $k=k_N\gg a_N^{-[\frac d2-\frac{2\alpha}{\alpha+2}]_+-\frac \alpha{\alpha+2}+o(1)}$, then the first terms on the right-hand sides dominate, and we obtain a stretched-exponentially decay.
\end{enumerate}
\end{proposition}

\begin{proof}

Recall that ${\mathfrak N}=\sum_{j=1}^N j X_j$ and $\Ns=\sum_{j=1}^{T_N} j X_j$ , where $X_1,\dots,X_N$ are independent Poisson random variables with parameters $\frac 1j t_{j,a }$, $j\in[N]$, where we recall the definition of $t_{j,a}$ from \eqref{tjadef}. In both proofs, we are going to use the exponential Chebychev inequality.
We are going to explicitly handle only the upwards deviations (i.e., for ${\mathfrak N}-\Ett[{\mathfrak N}]$ instead of $|{\mathfrak N}-\Ett[{\mathfrak N}]|$), since the case of the downwards deviations is similar. The  first term on the right stems from the application of Markov's inequality, and the second term from estimating the exponential expectation as follows.

(1)  For any $s\in(0,\infty)$,
    \begin{equation}
        \Ett\ek{\ex^{s \rk{{\mathfrak N}-\Ett\ek{{\mathfrak N}}}}}=\exp\rk{\sum_{j=1}^{N}\frac{1}{j}\rk{\ex^{s j}-1-sj}t_{j,a}}\, .
    \end{equation}
    We now pick $s=\kappa a_N^{\alpha/(\alpha+2)}$ and estimate the right-hand side. For the sum on $j\geq T_N$, we have $j a_N^{\alpha/(\alpha+2)}\to\infty$ and therefore get from Lemma~\ref{lem-tasy}(3), with some $C\in(0,\infty)$ that does not depend on $N$,
    \begin{equation}\label{eq:2.31}
    \begin{aligned}
        \sum_{T_N \leq j\leq N}\frac{1}{j}\rk{\ex^{s j}-1-sj}t_{j,a}
        &\le C\sum_{j\ge T_N}\frac{1}{j}\ex^{s j}\ex^{-\beta j a_N\lambda_1(w/a_N)}
        \leq \frac{C}{T_N}\sum_{j\geq T_N}\ex^{-a_N^{\alpha/(\alpha+2)} j[\beta\lambda_1(W)(1+o(1))-\kappa]}\\
        &\le \frac{C a_N^{\alpha/(\alpha+2)}}{\sqrt{\log\frac 1{a_N}}}\frac{\ex^{-cT_Na_N^{\alpha/(\alpha+2)}}}{ca_N^{\alpha/(\alpha+2)}}\leq C\ex^{-c\log(1/a_N)^{1/2}},
        \end{aligned}
    \end{equation} 
    since $\beta \lambda_1(W)>\kappa$. (If $\lim_{N\to\infty}a_N=0$ then it vanishes as $N\to\infty$.) 
    The sum on small $j$ is bounded  as follows. We use Lemma~\ref{lem-tasy}(1) and that $\ex^{x}-1-x\leq x^2\ex^{x}$ for any $x\in(0,\infty)$. Then we see that
    \begin{equation}\label{upperboundconcentration}
    \begin{aligned}
        \sum_{j\le T_N}\frac{1}{j}\rk{\ex^{s j}-1-sj}t_{j,a}
        &\le 
        C a_N^{-d/2}\sum_{j\leq T_N} j^{-1-d/2}s^2 j^2\ex^{sj}
        \leq C a_N^{2\alpha/(\alpha+2)-d/2}\ex^{\kappa T_N a_N^{\alpha/(\alpha+2)}}\sum_{j\le T_N}j^{1-d/2-d/\alpha}\\
&\leq   a_N^{2\alpha/(\alpha+2)-d/2-o(1)}\times\rk{1+T_N^{2-d/2-d/\alpha}+\log(T_N)},
        \end{aligned}
    \end{equation}
    where $a_N^{o(1)}$ is an estimate for $\ex^{\kappa T_N a_N^{\alpha/(\alpha+2)}}$, and the bracket is a generous upper bound for the $j$-sum in the three cases that $1-d/2-d/\alpha$ is $<-1$, or $=-1$ or $>-1$. Now use that $T_N=a_N^{-\alpha/(\alpha+2)+o(1)}$ to see that the right-hand side of \eqref{upperboundconcentration} is equal to $a_N^{o(1)}$ if $\frac d 2<\frac{ 2\alpha}{\alpha+2}$ and is equal to $a_N^{2\alpha/(\alpha+2)-d/2+o(1)}$ otherwise.

(2) The conclusion follows directly from the estimate of \eqref{upperboundconcentration} in (1). Since we no longer need \eqref{eq:2.31}, there is no restriction on $\kappa$.


    
\end{proof}

\subsection{Particles in long loops}\label{sec-longloops}

Recall that $\Nl=\sum_{T_N<j\leq N} j X_j$ is the number of particles in long loops in the PPP (recall \eqref{TNdef}). As in \cite{KVZ23}, we use now intricate combinatorial asymptotics to find sharp asymptotics for the asymptotic distribution of $\Nl$. Write $q\colon [0,\infty)\to [0,\infty)$ for the density of the random variable with Laplace transform
\begin{eqnarray}
    s\mapsto \exp\rk{\int_0^1 \rk{\frac{\ex^{-sx}}{x}-1}\d x}\, .
\end{eqnarray}
Note that $q(x)=\ex^{-\gamma}$ for $x\in[0,1]$,  where $\gamma\approx 0.5772$ is the Euler--Mascheroni constant. See \cite{arratia2003logarithmic} and Section~\ref{sec-PDproof} for more properties of $p$, in particular in connection with the Poisson--Dirichlet distribution.

\begin{lemma}\label{lem:2.7} Suppose that $w$ satisfies Assumption (W). For all sequences $(s_N)_{N\in\N}, (k_N)_{N\in\N}$ in $\N$ such that $T_N\ll s_N\leq k_N\le N$ for any $N$, and that $\lim_N \frac{k_N}{s_N}$ exists,
    \begin{equation}
        {\tt P}_{\beta a_N, w/a_N}^{\ssup N}\rk{\sum_{j=1+T_N}^{s_N}j X_j=k_N}\sim \frac{q(k_N/s_N) }{T_N}\ex^{-\beta a_N \l_1(w/a_N) k_N},\qquad N\to\infty.
    \end{equation}
    In particular,
    \begin{equation}
        {\tt P}_{\beta a_N, w/a_N}^{\ssup N}\rk{\Nl=k_N}\sim \frac{\ex^{-\gamma} }{T_N}\ex^{-\beta a_N\lambda_1(w/a_N) k_N},\qquad N\to\infty.
    \end{equation}
   
\end{lemma}

\begin{proof}
The proof follows the same argument as \cite[Proposition 2.7]{KVZ23}. 
Let us first consider the case $a_N\rightarrow 0$.
We write $\Pfrak_k^{\ssup {N}}$ for the set of sequences $m=(m_r)_{T_N<r\le s_N}$ of positive integers such that $\sum_{T_N<r\le s_N} r m_r=k_N$. 
Then 
    \begin{equation}\label{longloops1}
        {\tt P}_{\beta a_N, w/a_N}^{\ssup N}\rk{\sum_{j=T_N+1}^{s_N}j X_j=k_N}=\ex^{-\sum_{j=T_N+1}^{s_N}\frac{t_{j,a_N}}{j}}\sum_{m\in \Pfrak_k^{\ssup {N}}}\prod_{T_N<r\le s_N}\frac{t_{r,a_N}^{m_r}}{r^{m_r}m_r!}\, .
    \end{equation}
We claim that 
$
\ex^{-\sum_{j=T_N+1}^{s_N}\frac{t_{j,a_N}}{j}}\rightarrow 1
$ as $N\to\infty$. 
Indeed, since $a_N\to0$, we have $ja_N^{\alpha/(\alpha+2)}>T_Na_N^{\alpha/(\alpha+2)}\sim\sqrt{\log\frac 1{a_N}}\to\infty$. Therefore by Lemma~\ref{lem-tasy}(3) and Lemma~\ref{lem_spectralscaling},
\begin{equation}
\sum_{j> T_N}\frac{t_{j,a_N}}{j}\sim\sum_{j> T_N}\frac 1j\ex^{-\beta  \lambda_1(w/a_N) a_N j}\leq \frac 1{T_N}\frac 1{1-\ex^{-\beta\lambda_1(W) a_N^{\alpha/(\alpha+2)}}(1+o(1))}=O(\smfrac 1{a_N^{\alpha/(\alpha+2)} T_N})\to 0.
\end{equation}

For the remaining factor in \eqref{longloops1}, by Lemma~\ref{lem-tasy}(3), $t_{r,a_N}= \ex^{-\beta \lambda_1(w/a_N) ra_N}(1+\Ocal(\ex^{-\beta c r a_N^{\alpha/(2+\alpha)}}))$ for $r>T_N\rightarrow\infty$, hence, as in the proof of \cite[Proposition 2.7]{KVZ23}, we obtain
\begin{align}\label{eq:t_complex_lower}
\sum_{m\in \Pfrak_k^{\ssup{N}}}\prod_{T_N<r\le s_N}\frac{t_{r,a_N}^{m_r}}{r^{m_r}m_r!}
&\sim  
\ex^{-\beta \lambda_1(w/a_N) a_N k }\sum_{m\in \Pfrak_k^{\ssup{N}}}\prod_{T_N<r\le s_N}\frac{1}{r^{m_r}m_r!}\\
&=\ex^{-\beta \lambda_1(w/a_N) a_N k+\sum_{T_N<r\le s_N}\frac 1 r}\mathbb P\rk{\sum_{T_N<r\le s_N}rY_r=k_N},
\end{align}
where the $Y_r$'s are independent Poisson random variables with parameter $\frac 1 r$.
By \cite[Theorem 4.13]{arratia2003logarithmic} (take $\theta=1,b=T_N\ll n=s_N\le m=k_N,y=1$),
\[
\mathbb P\rk{\sum_{T_N<r\le s_N}rY_r=k_N}\sim\frac {q(k_N/s_N)} {s_N}.
\]
Therefore, we may conclude that
\begin{align}
{\tt P}_{\beta a_N, w/a_N}^{\ssup N}\rk{\Nl=k_N}\sim\sum_{m\in \Pfrak_k^{\ssup{N}}}\prod_{T_N<r\le s_N}\frac{t_{r,a_N}^{m_r}}{r^{m_r}m_r!}
\sim  
\frac{q(k_N/s_N)}{T_N}\ex^{-\beta \lambda_1(w/a_N) a_N k}.
\end{align}

When $(a_N)_{N\in\N}$ is bounded, we are interested in 
\begin{equation}
{\tt P}_{\beta a_N, w/a_N}^{\ssup N}\rk{\sum_{\sqrt{\log N}<j\le s_N}jX_j=k_N}
=\ex^{-\sum_{\sqrt{\log N}<j\le s_N}\frac{t_{j,a_N}}{j}}\sum_{m\in \Pfrak_k^{\ssup{\sqrt{\log N}}}}\prod_{\sqrt{\log N}<r\le s_N}\frac{t_{r,a_N}^{m_r}}{r^{m_r}m_r!},
\end{equation}
where we still have
\[
\sum_{j>\sqrt{\log N}}\frac{t_{j,a}}{j}
\sim\sum_{j>\sqrt{\log N}}\frac 1 j \ex^{-\beta \lambda_1(w/a_N) a_N j}\rightarrow 0,
\]
and
\begin{align*}
\sum_{m\in \Pfrak_k^{\ssup{\sqrt{\log N}}}}\prod_{\sqrt{\log N}<r\le s_N}\frac{t_{r,a_N}^{m_r}}{r^{m_r}m_r!}
\sim\ex^{-\beta \lambda_1 a_N k +\sum_{\sqrt{\log N}<r\le k}\frac 1 r}\frac{q(k_N/s_N)}{k}\sim \frac{q(k_N/s_N)}{\sqrt{\log N}}\ex^{\beta\lambda_1(w/a_N) a_Nk_N}.
\end{align*}
\end{proof}

\subsection{Lower bound for the denominator}\label{sec-lowbound}

We suppose that Assumption (W) holds. On base of Lemma~\ref{lem:2.7}, we give now a sharp lower bound for the denominator in \eqref{PPPreprfree}.

\begin{lemma}\label{lem-lowbounddenom} Assume that $\liminf_{N\to\infty}Na_N^{d/2}>\rho_w$, then there is a sequence $(\delta_N)_N$ that vanishes as $N\to\infty$ such that, for all large $N$, 
\begin{equation}\label{Eq:lowerBounded}
{\tt P}_{\beta a_N,w/a_N}^{\ssup N}(\Nrd=N)\geq\ex^{-\beta a_N  \lambda_1(w/a_N) N\rk{1-\rho_w/(N a_N^{d/2})+\delta_N}}(1+o(1)).
\end{equation}
\end{lemma}
\begin{proof}
     Abbreviate ${\tt P}_{\beta a_N,w/a_N}^{\ssup N}$ by $\tt P$, analogously for the expectations, and $a_N$ by $a$. Recall that $\mathfrak N=\sum_{j=1}^N j X_j$  and that $\Ns=\sum_{j\leq T_N} j X_j$, where the $X_j$ are independent Poisson  random variables under $ \P$ with  parameters $\frac 1j t_{j,a}$.
We lower bound against the event that there is one large loop and otherwise only small ones with about $\rho_w a^{-d/2}$ particles:
    \begin{equation}
    \begin{aligned}
       {\tt P}\rk{\Nrd=N}&\ge \sum_{k\in\N\colon |k- \rho_wa^{-d/2}|\leq \delta_N N}{\tt P}\rk{X_{N-k}=1}{\tt P}\rk{X_j=0\text{ for all }j\in \gk{T_N,\ldots, N-1}\setminus\{N-k\}}\\
        &\quad\times{\tt P}\rk{\Ns=k}\,,
        \end{aligned}
    \end{equation} 
    where $\delta_N\in(0,1)$ with $1\gg \delta_N\gg 1/N$ is suitable (see below). 
    For all $k$ in that sum, we have
    \begin{equation}
        {\tt P}\big(X_j=0\text{ for all }j\in \gk{T_N,\ldots, N-1}\setminus\{N-k\}\big)\ge \exp\rk{-\sum_{j\ge T_N}\frac{t_{j,a}}{j}}=1+o(1)\, 
    \end{equation}
    as we saw in the proof of Lemma~\ref{lem:2.7}. Furthermore, by Lemma~\ref{lem-tasy}(3),
    \begin{equation}
        {\tt P}\rk{X_{N-k}=1}=t_{N-k,a}\ex^{-t_{N-k,a}}\sim\ex^{-\beta a\l_1(w/a) (N-k) }\geq \ex^{-\beta a\l_1(w/a)N(1-\rho_w /(Na^{d/2})+\delta_N) }\, .
    \end{equation}
Recall from Corollary~\ref{cor:N_short} that ${\tt E}\ek{\Ns}\sim \rho_wa^{-2/d}$. Using Proposition \ref{prop:2.6} for $k=\delta_N N$ with $\delta_N$ picked such that the first term on the right-hand side of \eqref{shortloopsconc} is the leading term, we get that 
$$
\sum_{k\in\N\colon |k- \rho_wa^{-d/2}|\leq \delta_N N}{\tt P}\rk{\Ns=k}\geq{\tt P}\big( |\Ns-{\tt E}[\Ns]|\leq \smfrac 12 \delta_N N\big)\to 1,\qquad  \mbox{as }N\to\infty.
$$
This implies \eqref{Eq:lowerBounded}.
\end{proof}

\section{Proof of Theorem~\ref{thm-ODLROlongloops}(1): super-critical regime}\label{sec-ProofODLRO}

\noindent This section is under the assumption that $\chi=\liminf_{N\to\infty}N a_N^{d/2}>\rho_w$ and contains the proof of Theorem~\ref{thm-ODLROlongloops}(1), i.e.,  for the asymptotics \eqref{gammaasy} of the reduced one-particle density matrix in Section~\ref{sec-ODLROproof},  for the limiting distribution of the macroscopic loop lengths in terms of the Poisson--Dirichlet distribution in Section~\ref{sec-PDproof} and for the convergence of the normalized PPP (i.e., the microscopic loop lengths) in Section~\ref{sec-microconvProof}. (The proof of Theorem~\ref{thm-ODLROlongloops}(1)(d) is deferred to Section~\ref{sec-freeenergy}.) As always, we are under Assumption (W) for the trap potential $w$. Recall that $\rho_w=\sum_{k\in\N}(4\pi\beta k)^{-d/2}\int\ex^{-\beta k w}\d x$.

\subsection{Proof of \texorpdfstring{\eqref{gammaasy}}{} in
Theorem~\ref{thm-ODLROlongloops}(1)}\label{sec-ODLROproof}

This proof is analogous to the proof of \cite[Proposition 2.1]{KVZ23}. We abbreviate $a_N$ by $a$ and  ${\tt P}_{\beta a,w/a}^{\ssup N}$ by ${\tt P}$, analogously for the expected value. Our starting point is the representation of $\gamma_N^{\ssup a}$ from Lemma~\ref{Cor-PPPreprfree}, that is, 
\begin{equation}\label{startingpoint}
    \gamma^{\ssup a}_N(x,y)=\sum_{r=1}^N \BB_{x,y}^{\ssup{\beta a r,w/a}}(\Ccal_{\beta a r})
    \frac{{\tt P}(\Nrd=N-r)}
    { {\tt P}(\Nrd=N)}.
    \end{equation}

We carry out the proof only for the case $a_N\to0$ as $N\to$ and leave the second case to the reader. Fix some small $\eps>0$. It is not hard to show that in \eqref{startingpoint}, the two partial sums on $r\leq T_N$ and on $r>N(1-\frac {\rho_w}\chi-\eps)$ are negligible by using the estimate $\BB_{x,y}^{\ssup{\beta,w}}(\Ccal_{\beta})\le\frac{1}{(4\pi \beta)^{d/2}}\ex^{-\abs{x-y}^2/(4\beta)}$ and the lower bound for ${\tt P}(\mathfrak N=N)$ from Lemma~\ref{lem-lowbounddenom}.

For the remaining, we decompose the number $\mathfrak N$ of all particles into $\mathfrak N=\Ns+\Nl$, which denote the number of particles in loops of lengths $\leq T_N=\lfloor a_N^{-\alpha/(\alpha+2)}\log(\frac 1{a_N})^{1/2}\rfloor$ respectively of lengths $>T_N$; see \eqref{TNdef}. Then 
\begin{equation}\label{eq:supercritical_N}
{\tt P}(\Nrd=N-r)=\sum_k{\tt P}(\Nrd^{\ssup{\rm short}}=k){\tt P}(\Nrd^{\ssup{\rm long}}=N-r-k).
\end{equation}
We observe from Corollary~\ref{cor:N_short} that 
$$\limsup_{N\to\infty}\frac {\Ett\ek{\Ns}}N =\rho_w \limsup_{N\to\infty}\frac1{N a^{d/2}}=\frac{\rho_w} \chi<1,
$$ in the case of Theorem~\ref{thm-ODLROlongloops}(1). 
According to Proposition~\ref{prop:2.6}, the sum on $k$ strongly concentrates around the expectation 
$$
{\tt E}(\Nrd^{\ssup{\rm short}})\sim \rho_w a^{-d/2},
$$
more precisely, to estimate \eqref{eq:supercritical_N}, we can focus on $k\in[\rho_w a^{-d/2}-\eps N,\rho_wa^{-d/2}+\eps N]\cap\N$ for all sufficiently large $N$. 

Furthermore, according to Lemma~\ref{lem:2.7},
\begin{equation}\label{convol}
{\tt P}(\Nrd^{\ssup{\rm long}}=N-r-k)\sim \frac{\ex^{-\gamma}}{T_N}\ex^{-\beta a \lambda_1(w/a)(N-r-k)},
\end{equation}
as long as $r\ll N-k\leq  N-\rho_w a^{-d/2}-\eps N\sim N (1-\rho_w/\chi-\eps)$. 
Using \eqref{convol} once more for $N-k$ instead of $N-r-k$, we see that
\begin{equation}\label{Heur1}
{\tt P}(\Nrd^{\ssup{\rm long}}=N-r-k)\sim {\tt P}(\Nrd^{\ssup{\rm long}}=N-k)\ex^{\beta a \lambda_1(w/a) r}.
\end{equation}

Finally, from Lemma~\ref{lem-tasy}(3) we deduce that
\begin{equation}\label{eq:34}\BB_{x,y}^{\ssup{\beta ar,w/a}}(\Ccal_{\beta ar})\sim \ex^{- \beta ar\lambda_1(w/a)}\phi_1^{\ssup {w/a} }(x)\phi_1^{\ssup {w/a}}(y)\qquad \text{ if }ra^{\alpha/(\alpha+2)}\to\infty. 
\end{equation}

Putting \eqref{eq:supercritical_N}, \eqref{Heur1} and \eqref{eq:34} into \eqref{startingpoint}, we have
$$
\begin{aligned}
\gamma^{\ssup a}_N(x,y)&\sim 
\sum_{r=T_N}^{N(1-\rho_w/\chi-\eps)} \ex^{-\beta a r\lambda_1(w/a)}\phi_1^{\ssup a}(x)\phi_1^{\ssup a}(y)\ex^{\beta a \lambda_1(w/a) r}\\
&\qquad \times \frac{\sum_{k\colon |k-\rho_w a^{-d/2}|\leq \eps N}{\tt P}(\Nrd^{\ssup{\rm short}}=k){\tt P}(\Nrd^{\ssup{\rm long}}=N-k)}{ {\tt P}(\Nrd=N)}\\
&\sim\sum_{r=T_N}^{N(1-\rho_w/\chi-\eps )} \phi_1^{\ssup a}(x)\phi_1^{\ssup a}(y)\\
&=N\Big(1-\frac{\rho_w}\chi-\eps-T_N\Big) \phi_1^{\ssup a}(x)\phi_1^{\ssup a}(y)(1+o(1)).
\end {aligned}
$$
Now the conclusion follows by noticing $T_N=o(N)$ and taking $\eps\downarrow0$.

\subsection{Convergence to the Poisson--Dirichlet distribution}\label{sec-PDproof}

In this section, we prove Theorem \ref{thm-ODLROlongloops}(1)(b). Recall that $L_1\geq L_2\geq L_3\geq ...$ are the lengths appearing in the loop soup. Recall the density $q$ introduced before Lemma~\ref{lem:2.7}.
Our main goal is then reduced to the following:
\begin{proposition}\label{PropPoissonDirConv}
    Suppose that $\chi\in(\rho_w,\infty]$. Then, for any $m\in\N$ and $t_1>\ldots> t_m>0$ with $\sum_{i=1}^m t_i<1$,
    \begin{multline}\label{PDweakconv}
        \Ptt_{\beta a_N,w/a_N}^{\ssup N}\rk{\smfrac 1{N(1-\rho_w/\chi)} \big(L_1,\dots,L_m\big)\in \d (t_1,\ldots,t_m)\,\Big|\,\Nrd=N}\\
        \Longrightarrow \frac{\ex^{\gamma}}{t_1\cdots t_m}q\rk{\frac{1-(t_1+\ldots+t_m)}{t_m}}\,\d(t_1,\ldots,t_m) .
    \end{multline}
\end{proposition}

From this, the weak convergence of $(N(1-\rho_w/\chi))^{-1}(L_i)_{i=1,\dots,m}$ towards the first $m$-dimensional distribution of the Poisson--Dirichlet distribution follows, according to the Portemanteau theorem. From Scheff\'{e}'s theorem, see~\cite[Corollary 5.11]{arratia2003logarithmic}, the convergence of the  entire sequence  follows.

\begin{proof} Abbreviate ${\tt P}=\Ptt_{\beta a_N,w/a_N}^{\ssup N}$ and $a=a_N$.
    Fix $j_1\ge j_2\ge \ldots\ge j_m\in\N$ that such that $j_i\sim t_{i} N(1-\rho_w/\chi)$, for all $1\le i\le m$. Then, for all large $N$, we even have that $j_1>j_2>\ldots>j_m$. Abbreviate $A=\{L_1=j_1,\ldots,L_m=j_m \}$. Recall that ${\rm N}^{\ssup{\rm long}}$ denotes the number of particles in long loops, i.e., in loops of length $>T_N$ defined in \eqref{TNdef}. Using  the concentration result of Proposition~\ref{prop:2.6} and the lower bound in  Lemma~\ref{lem-lowbounddenom}, we can  decompose
$$
\Ptt(A\mid{\mathfrak N}=N)
=\sum_{k \in\N\colon |\frac kN-(1-\rho_w/\chi)|\leq  \delta_N} \Ptt(A\mid{\mathfrak N}^{\ssup{\rm long}}=k)\Ptt({\mathfrak N}^{\ssup{\rm long}}=k\mid {\mathfrak N}=N)+o(N^{-m}),
$$
where $(\delta_N)_N$ is as in Lemma~\ref{lem-lowbounddenom}, i.e., it satisfies $ \delta_N\to0$.
As in the proof of  \cite[Proposition 4.1]{KVZ23}, it suffices to show that, for any $k=k_N$ in the sum above,
\begin{equation}\label{PDproofgoal}
\lim_{N\to\infty}(N(1-{\smfrac{\rho_w}\chi}))^{m}\Ptt(A\mid{\mathfrak N}^{\ssup{\rm long}}=k_N)=\frac{\ex^{\gamma}}{t_1\cdots t_m}q\rk{\frac{1-(t_1+\ldots+t_m)}{t_m}}.
\end{equation}

Recall that $X_l$ is equal to the number of loops of length $l$ and that all the $X_l$ are independent under $\Ptt$. We then have that
    \begin{eqnarray}
        \Ptt\rk{A}=\prod_{l=j_m}^N\Ptt\rk{X_l=i_l}\,\qquad\mbox{where }i_l=\#\{k\colon j_k=l\}\in\{0,1\} \mbox{ for all }l.
    \end{eqnarray}

Similarly, for $k\ge J$, where $J=\sum_{i=1}^m j_i$,
    \begin{equation}\label{goal}
        \Ptt\rk{A\mid\Nl=k}=\frac{\Ptt\rk{\sum_{i=1+T_N}^{j_m-1}iX_i=k-J}}{\Ptt\rk{\sum_{i=1+T_N}^{N}iX_i=k}}\prod_{l=j_m}^N\Ptt\rk{X_l=i_l}\, .
    \end{equation}
Note that $i_l=0$ if $l\notin\{j_1,\dots,j_m\}$ and $=1$ otherwise.    Using the approximation $t_{j,a}\sim \ex^{-\beta a j \l_1(w/a)}\to 0$ (see Lemma~\ref{lem-tasy}(3)) for $j\in\{j_1,\dots,j_m\}$, we get that
    \begin{equation}
    \begin{aligned}
     \prod_{l=j_m}^N\Ptt\rk{X_l=i_l}&=
        \prod_{l=j_m}^N\ex^{-t_{l,a}}\frac{(t_{l,a})^{i_l}}{i_l!l^{i_l}}\sim \exp\rk{-\beta \sum_{l=j_m}^N i_l l a \l_1(w/a)}\prod_{l=j_m}^N\frac{1}{i_l!l^{i_l}}= \ex^{-\beta a \l_1(w/a) J}\prod_{i=1}^m\frac{1}{j_i}\\
        &\sim \ex^{-\beta a \l_1(w/a) J}\big(N(1-{\smfrac{\rho_w}\chi})\big)^{-m}\prod_{i=1}^m\frac{1}{t_i}.
  \end{aligned}
    \end{equation}
 Now pick $k=k_N\sim N (1-\rho_w/\chi)$, we obtain by Lemma~\ref{lem:2.7}
    \begin{eqnarray}
        \Ptt\rk{\sum_{i=1+T_N}^{N}jX_j=k_N}\sim \frac{\ex^{-\gamma}}{T_N}\ex^{-\beta a\l_1(w/a) k_N}\, ,
    \end{eqnarray}
    as well as (observe that $(k_N-J)/j_m\to (1-(t_1+\dots+t_m))/t_m$ as $N\to\infty$)
    \begin{eqnarray}
        \Ptt\rk{\sum_{i=T_N}^{j_m-1}jX_j=k_N-J}=\frac{q\rk{(1-(t_1+\dots+t_m))/t_m }}{T_N}\ex^{-\beta a\l_1(w/a) (k_N-J)}\, .
    \end{eqnarray}
Substituting the last three displays in \eqref{goal} implies \eqref{PDproofgoal}, and we finish the proof.
\end{proof}

\subsection{Proof of convergence of \texorpdfstring{$\frac 1N(iX_i)_{i\in\N}$}{}}\label{sec-microconvProof}
In this section, we prove Theorem~\ref{thm-ODLROlongloops}(1)(c), i.e., the convergence of the distribution of the microscopic loop lengths. Since we are considering the product topology, it suffices to consider just $\frac 1N i X_i$ for one fixed $i\in\N$. Recall that $u_\chi=0$. By Lemma~\ref{lem-tasy}(2), $X_i$ is Poisson-distributed with parameter $\frac 1i t_{i,a_N}\sim a_N^{-d/2} \frac 1i\chi \alpha^{\ssup\chi}_i\sim N\frac 1i \alpha^{\ssup\chi}_i$ as $N\to\infty$. For any $\eps>0$,
$$
\Ptt_{\beta a_N,w/a_N}^{\ssup N}\Big(\Big|\frac 1N i X_i -\alpha^{\ssup\chi}_i\Big|>\eps\,\Big|\,{\mathfrak N}=N\Big))
\leq \Ptt_{\beta a_N,w/a_N}^{\ssup N}\Big(\Big|\frac 1N i X_i -\alpha^{\ssup\chi}_i\Big|>\eps\Big) \,\frac1{\Ptt_{\beta a_N,w/a_N}^{\ssup N}({\mathfrak N}=N)}.
$$
Observe that $X_i$ is distributed as a sum of $N$ independent Poisson-distributed random variables with parameter $\frac 1i \alpha^{\ssup\chi}_i(1+o(1))$. Use a standard exponential concentration inequality based on Cram\'er's theorem from the theory of large deviations, we conclude that the first term on the right-hand side vanishes exponentially small on the scale $N$. On the other side, we use the lower bound of Lemma~\ref{lem-lowbounddenom} and the asymptotics from Lemma~\ref{lem_spectralscaling} to see that the denominator vanishes exponentially fast on the scale $N a_N \lambda_1(w/a_N)\asymp N a_N^{\alpha/(\alpha+2)}\ll N$. 	Hence, the right-hand side decays exponentially fast on the scale $N$.

\section{Proof of Theorem~\ref{thm-ODLROlongloops}(2): sub-critical regime}\label{sec-proofthm(2)}

\noindent Abbreviate $\chi_N=N a_N^{d/2}$. In this section we are under the assumption that $\chi=\lim_{N\to\infty}\chi_N$ exists in $ [0,\rho_w)$, and we prove Theorem~\ref{thm-ODLROlongloops}(2)(a) and (b). (The proof of (c) is deferred to Section~\ref{sec-freeenergy}.)

Abbreviate $a=a_N$ and ${\tt P}={\tt P}_{\beta a, w/a}^{\ssup N}$. Since
\begin{equation}
   \frac 1N \Ett\ek{\Nrd}\sim\frac 1N\rho_w a^{-2/d}\to\frac{\rho_w}{\chi_N}> 1\, ,
\end{equation}
the event $\{{\mathfrak N}=N\}$ is a downwards deviation under ${\tt P}$. We tilt the intensity measure of ${\tt P}$ with a small factor by means of a \textit{chemical potential}, which suppresses long loops, such that the expected number of particles in the process is equal to $N$.  For $\mu\in(-\infty,0)$, denote by $\Ptt_{\beta,w, \mu}^{\ssup N}$ the probability measure for the PPP with intensity measure
\begin{equation}
    \nu_{\beta, w,\mu}^{\ssup N}(\d f)=\sum_{k=1}^N \frac{\ex^{\beta \mu k}}{k} \BB^{\ssup{k \beta, w}}(\d f),\qquad \mbox{on }\bigcup_{k\in\N}\Ccal_{\beta k}.
\end{equation}
Abbreviate  ${\tt P}_\mu={\tt P}_{\beta a, w/a,\mu}^{\ssup N}$. Under ${\tt P}_\mu$, the vector $(X_j)_{j\in[N]}$ consists of independent Poisson-distributed variables $X_j$ with parameters $\frac 1j t_{j,a}^{\ssup\mu}= \frac 1j\ex^{\beta \mu a j}t_{j,a}$.
 Observe that 
\begin{align}\label{changeofmeas}
\Ptt(\cdot\mid{\mathfrak N}=N)
=\Ptt_\mu(\cdot\mid {\mathfrak N}=N),\qquad N\in\N, \mu\in(-\infty,0),
\end{align}
since a simple change of measure shows that
\begin{equation}\label{change}
    \Ptt\rk{\Nrd=m}=\ex^{p_{a,N }(\mu)-p_{a,N}(0)-\beta \mu a m}\Ptt_{\mu}\rk{\Nrd=m}\, ,\qquad m\in\N,
\end{equation}
where we abbreviated
\begin{equation}\label{paNdef}
 p_{a,N}(\mu)=\nu_{\beta a, w/a,\mu}^{\ssup N}\Big(\bigcup_{j=1}^N\Ccal_{\beta j}\Big)= \sum_{j=1}^N\frac{\ex^{\beta \mu a j}}{j} t_{j,a}.
\end{equation}
Now we define $\mu_N\in(-\infty,0)$ by ${\tt E}_{\mu_N}[{\mathfrak N}]=N$. Recall the pressure $p$ from \eqref{rhowdef} and \eqref{pressuredef} and that $u_\chi\in(-\infty,0)$ is defined by $ p'(u_\chi)=\beta\chi$.

\begin{lemma}\label{lem:b}
$$
\lim_{N\to\infty}\mu_Na_N=\begin{cases}
u_\chi,&\mbox{if }\chi>0,\\
-\infty,&\mbox{if }\chi=0,
\end{cases}.
$$ 
In the case $\chi=0$, we have the more precise asymptotics $\mu_Na_N\sim \frac{1}{\beta}\log\rk{\chi_N(4\pi\beta)^{d/2}\Wj_1^{-1}}+o(1)$.
\end{lemma}

\begin{proof}
Note that 
$$
\mathtt{E}_\mu[{\mathfrak N}]=\sum_{j=1}^N \ex^{\beta \mu a j}t_{j,a}.
$$
Since this is equal to $N$ for $\mu=\mu_N$, we see that  $(\mu_N a_N)_{N\in\N}$ is bounded away from zero. Indeed, if $\mu_N a_N$ would go to zero, then we would have, for any $R\in\N$, using Lemma \ref{lem-tasy}(2)
$$
N\geq\sum_{j=1}^R \ex^{\beta \mu_N a_N j}t_{j,a}
\geq (1-o(1))\sum_{j=1}^R a^{-d/2} (4\pi\beta j)^{-d/2}\int \ex^{-\beta j w}\d x\sim \frac N\chi \sum_{j=1}^R (4\pi\beta j)^{-d/2}\int \ex^{-\beta j w}\d x,
$$
and the right-hand side is asymptotic to $N\rho_w/\chi$ in the limit $N\to\infty$, followed by $R\to\infty$, which produces a contradiction with $\chi<\rho_w$. Using Lemma \ref{lem-tasy}(1) and (2) and the fact that $d\geq 3$, we see that 
$$
1=\frac 1N \mathtt{E}_{\mu_N}[{\mathfrak N}]\sim \frac 1N a_N^{-d/2} p(\mu_N a_N)
\sim \frac{1}\chi  p(\mu_N a_N), \qquad N\to\infty.
$$
This concludes the proof for $\chi>0$, since the range of $p$ contains $(0,\rho_w]$.

In the case $\chi=0$, note that
\begin{equation}
p(u)=\ex^{\beta u}(4\pi \beta)^{-d/2}\Wj_1+\Ocal\rk{\ex^{2\beta u}}\, , \qquad u\to-\infty.
\end{equation}
Hence, 
\begin{equation}
\mu_Na_N\sim \frac{1}{\beta}\log\rk{\chi_N(4\pi \beta)^{d/2}\Wj_1^{-1}}
\sim\frac{1}{\beta}\log \chi_N, \qquad N\to\infty.
\end{equation}

%
%
%
%

\end{proof}

Write ${\mathfrak N}^{\ssup j}=j X_j$ for the total number of particles in loops of length $j$ and ${\mathfrak N}^{\ssup{\ge j}}=\sum_{k=j}^\infty k X_k$ for the number of all particles in loops of lengths $\geq j$.



\begin{lemma}\label{lem:E&Var}
\begin{enumerate}
\item If $\chi>0$, we have that
\begin{equation*}
\Var_{\mu_N}\ek{\Nrd}\sim \frac{N}{\chi}\,\frac{p''(u_\chi)}{\beta^2}, \qquad N\to\infty.
\end{equation*}
\item If $\chi_N\to 0$, there is a $C\in(0,\infty)$ such that for any $R,N\in\N $ with $R\leq N$,
\begin{eqnarray}
\Ett_{\mu_N}[{\mathfrak N}^{\ssup{1}}]&\sim& N,\nonumber\\
\Ett_{\mu_N}[{\mathfrak N}^{\ssup{\geq R}}]&\leq& C N \chi_N^{R-1}\label{ENResti},\\
 {\tt Var}_{\mu_N}[{\mathfrak N}^{\ssup{\geq 2}}]&\leq&C N \chi_N^{\frac 12}.\label{eq:geR}
\end{eqnarray}
\end{enumerate}
\end{lemma}

\begin{proof} (1) Note that $p''(u)/\beta^2=\sum_{j\in\N}\ex^{\beta u j}(4\pi\beta j)^{-d/2}\Wj_j$, since $p''$ is continuous in $(-\infty,0)$. Since $(\mu_N a_N)_{N\in\N}$ is bounded away from zero, we can use for any $j\in [N ]$ the asymptotics $t_{j,a}\sim (4\pi \beta)^{-d/2}\Wj_j N/\chi_N$ in the following sum:
\begin{equation}
\Var_{\mu_N}[{\mathfrak N}]=\sum_{j=1}^N \ex^{\beta \mu_N a_N j}jt_{j,a_N}\sim \frac{N}{\chi_N}\frac{p''(\mu_N a_N)}{\beta^2}\sim \frac{N}{\chi}\frac{p''(u_\chi)}{\beta^2}\,.
\end{equation}

(2)  By Lemma \ref{lem-tasy}(2) and Lemma \ref{lem:b},
\[
\Ett_{\mu_N}[{\mathfrak N}^{\ssup{1}}]=\ex^{\beta\mu_N a_N}t_{1,a_N}\sim \ex^{\beta\mu_N a_N}a_N^{-d/2}\sim N.
\]

We use $C\in(0,\infty)$ to denote a generic constant that does not depend on $a$ nor on $N$ and may change its value at each appearance.
By Lemma \ref{lem-tasy}(1) and Lemma \ref{lem:b} again,
\begin{align*}
{\tt E}_{\mu_N}[{\mathfrak N}^{\ssup{\ge R}}]=\sum_{j=R}^N \ex^{\beta \mu_N a_N j}t_{j,a_N}
&\leq  C  a_N^{-d/2}\ex^{\beta \mu_N a_N R}\sum_{j=R}^N \ex^{\beta \mu_N a_N(j-R)} j^{-d/2}\Wj_j\\
&\leq C\frac N {\chi_N} \,\chi_N^{R(1+o(1))}\sum_{j=R}^N  j^{-d/2}\Wj_j\leq C N  \chi_N^{(R-1)(1+o(1)}.
\end{align*}
Finally, Lemma \ref{lem:b} implies that
\begin{equation}
\begin{aligned}
    {\tt Var}_{\mu_N}[{\mathfrak N}^{\ssup{\geq 2}}]&=\sum_{j=2}^N j\ex^{\beta \mu_N a_N j}t_{j,a_N} 
    \leq C a_N^{-d/2}\sum_{j=2}^Nj^{1-d/2}\ex^{\beta \mu_N a j}\Wj_j\\
& 
    \leq C a_N^{-d/2} \ex^{\frac 32 \beta \mu_N a_N}\sum_{j=1}^Nj^{1-d/2}\ex^{\frac 12 \beta \mu_N a_N j}\Wj_j   \leq C N \ex^{\frac 12\beta \mu_N a_N}=CN\chi_N^{\frac 12}\, .
\end{aligned}
\end{equation}

\end{proof}

\begin{lemma}\label{lem:almost_gaussian} There is $C\in(0,\infty)$ such that, for any $N\in\N$ and any $r=r_N\in\N_0$ such that $r\leq O(\sqrt N)$,
\begin{align}
C^{-1} N^{-\frac 1 2}\le \Ptt_{\mu_N}({\mathfrak N}=N-r)\leq C N^{-\frac 1 2}.
\end{align}
\end{lemma}

\begin{proof}
For the case $\chi>0$, this follows from the variance bound proven in Lemma \ref{lem:E&Var}. Indeed, ${\mathfrak N}$ is the sum of $N$-independent random variables with mean $N$ and variance $\Ocal(N)$, so the result follows from the local central limit theorem.
Below we consider $\chi=0$, which requires more approximations.

We first prove the lower bound. Let $s(N):=2\sqrt{{\tt{Var}}_{\mu_N}({\mathfrak N}^{\ssup{\geq 2}})}$. Recall that by Lemma \ref{lem:b}, ${\mathfrak N}^{\ssup{1}}$ and
${\mathfrak N}^{\ssup{\geq 2}}$ are independent and that ${\mathfrak N}^{\ssup{1}}$ has the Poisson distribution 
\[
\mathtt P_{\mu_N}({\mathfrak N}=k)={\rm Poi}_{\alpha}(k):=\frac{\ex^{-\alpha}\alpha^k}{k!},\qquad k\in\N_0,
\] 
where $\alpha:=\mathtt E_{\mu}[{\mathfrak N}^{\ssup{1}}]\sim N$, and $\mathtt E_\mu[{\mathfrak N}^{\ssup{\ge 2}}]=N-\alpha$. For $r=\Ocal(\sqrt{N})$, expand
\begin{equation}\label{eq:expand_N-r}
\begin{aligned}
\Ptt_{\mu_N}({\mathfrak N}=N-r)&=
\sum_{k\in\mathbb Z-\alpha} \Ptt_{\mu_N}({\mathfrak N}^{\ssup{1}}=\alpha+k-r)\Ptt_ {\mu_N}({\mathfrak N}^{\ssup{\geq 2}}=N-\alpha-k)\\
&=\sum_{k\in\mathbb Z-\alpha} {\rm Poi}_{\alpha}(\alpha+k-r)\Ptt_ {\mu_N}({\mathfrak N}^{\ssup{\geq 2}}=N-\alpha-k)\\
&\ge \Ptt_{\mu_N}\big(|{\mathfrak N}^{\ssup{\geq 2}}-\mathbb E_{\mu_N}[{\mathfrak N}^{\ssup{\geq 2}}]|\le s(N)\big)\min_{|k|\le s(N),k\in\Z-\alpha}{\rm Poi}_{\alpha}(\alpha+k-r).
\end{aligned}
\end{equation}
Using Stirling's formula in the form $n!\leq C (\frac n\ex)^n\sqrt n$, we estimate for $l=k-r$
\begin{equation*}
\begin{aligned}
{\rm Poi}_\alpha(\alpha+l)
&=\ex^{-\alpha}\frac{\alpha^{\alpha+l}}{(\alpha +l)!}
\geq C\ex^{-\alpha}\ex^{\alpha +l}\Big(\frac\alpha{\alpha+l}\Big)^{\alpha+l}(\alpha+l)^{-1/2}
\geq C \ex^{l} \ex^{-\frac l\alpha(\alpha+l)}N^{-1/2}\\
&\geq  C\ex^{-l^2/\alpha}N^{-1/2}\ge C N^{-1/2},
\end{aligned}
\end{equation*}
since $s(N)\leq \Ocal(\sqrt N)$ by Lemma~\ref{lem:E&Var}.

Finally, by Chebyshev's inequality,
\begin{align}
\Ptt_{\mu_N}\big(|{\mathfrak N}^{\ssup{\geq 2}}-\mathbb E_{\mu_N}[{\mathfrak N}^{\ssup{\geq 2}}]|\le s(N)\big)
\ge 1-\frac{{\tt Var}_{\mu_N}({\mathfrak N}^{\ssup{\geq 2}})}{s(N)^2}=\frac 3 4,
\end{align}
and the claimed lower bound for $\Ptt_{\mu_N}({\mathfrak N}=N-r)$ follows.

For the upper bound, simply notice that by \eqref{eq:expand_N-r},
\begin{equation*}
\begin{aligned}
\Ptt_{\mu_N}({\mathfrak N}=N-r)
&=\sum_{k\in\mathbb Z-\alpha} {\rm Poi}_{\alpha}(\alpha+k-r)\Ptt_ {\mu_N}({\mathfrak N}^{\ssup{\geq 2}}=N-\alpha-k)\\
&\le\sup_{k\in\mathbb N} {\rm Poi}_{\alpha}(k)={\rm Poi}_{\alpha}([\alpha])\le CN^{-1/2}.
\end{aligned}
\end{equation*}
\end{proof}

\begin{proof}[Proof of Theorem~\ref{thm-ODLROlongloops}(2)] Recall that we are in the case where $\chi=\lim_{N\to\infty}N a_N^{d/2}\in[0,\rho_w)$. Recall that $\mu_N\in(-\infty,0)$ is picked such that ${\tt E}_{\mu_N}[\mathfrak N]=N$. By Lemma \ref{Cor-PPPreprfree},
\begin{equation}
    \gamma^{\ssup {a_N}}_N(x,y)=\sum_{r=1}^N \ex^{\beta \mu_N a_N r} \BB_{x,y}^{\ssup{\beta a_N r,w/a_N}}(\Ccal_{\beta a_N r})
    \frac{{\tt P}_{\mu_N}(\Nrd=N-r)}
    { {\tt P}_{\mu_N}(\Nrd=N)}\, .
\end{equation}
We split the sum into the sums on $r\leq \sqrt N$, where we will use that the $\BB$-term is small for all distinct $x,y$, and $r>\sqrt N$, where we will use that the exponential term is small.
Using  Lemma~\ref{lem:almost_gaussian} for both the numerator and the denominator and using the simple bound $\BB_{x,y}^{\ssup {t,w}}\leq C t^{-d/2}\ex^{-|x-y|/4t}$, we obtain
\begin{equation}\label{subcaseproof1}
\begin{aligned}
    \sum_{1\leq r\leq \sqrt{N}} &\ex^{\beta \mu_N a_N r} \BB_{x,y}^{\ssup{\beta a_N r,w/a_N}}(\Ccal_{\beta a_N r})
    \frac{{\tt P}_{\mu_N}(\Nrd=N-r)}
    { {\tt P}_{\mu_N}(\Nrd=N)}
    \le C a_N^{-d/2}\sum_{1\le r\leq \sqrt{N}}r^{-d/2} \ex^{\beta \mu_N  a_N r} \ex^{-|x-y|^2/(4\beta a_N r)}.
\end{aligned}
\end{equation}
We use the comparison between geometric and arithmetic mean ($\frac{a+b}2\geq \sqrt {ab}$) to see that 
\[\ex^{\frac 12\beta \mu  a r} \ex^{-|x-y|^2/(4\beta ar)}
\le \ex^{-|x-y|\,(|\mu|/2)^{1/2}}.
\]
Since  $\mu_N a_N\to u_\chi<0$, respectively $\to-\infty$ for $\chi=0$, we find a $c\in(0,\infty)$ such that $|\mu_N|/2\geq c^2/a_N$ for all $N$. This implies that the sum on $r\leq \sqrt N$ is not larger than the right-hand side of \eqref{gammabound}.

In the remaining sum, we can use Lemma~\ref{lem:almost_gaussian} only for the denominator, but analogously we obtain in the same way
$$
\begin{aligned}
\sum_{ \sqrt{N}<r\leq N} \ex^{\beta \mu_N a_Nr}& \BB_{x,y}^{\ssup{\beta a_N r,w/a_N}}(\Ccal_{\beta a_N r})
    \frac{{\tt P}_{\mu_N}(\Nrd=N-r)}
    { {\tt P}_{\mu_N}(\Nrd=N)}
    \le C a_N^{-d/2}\sqrt N\sum_{\sqrt N< r\leq N}r^{-d/2}  \ex^{\frac 12\beta \mu  a_N r}\ex^{-|x-y|\,(|\mu_N|/2)^{1/2}}\\
    &\leq C a_N^{-d/2}\sqrt N\ex^{\frac 14\beta \mu_N  a_N \sqrt N} \ex^{-|x-y|\,(|\mu_N|/2)^{1/2}}
    \leq o(a^{-d/2})\ex^{-|x-y|\,(|\mu_N|/2)^{1/2}}.
    \end{aligned}
$$
Hence, this part is even smaller than the sum on $r\leq \sqrt N$, which finishes the proof of \eqref{gammabound}.

Now we  prove the weak convergence of $\frac 1N (i X_i)_{i\in\N}$ under ${\tt P}$ towards $\alpha=\alpha^{\ssup\chi}$ defined in \eqref{alphadef}. First we assume that $\chi>0$. Observe that $\alpha^{\ssup\chi}_j=\lim_{N\to\infty}{\tt E}_{\mu_N}(j X_j)$ for any $j\in\N$. Hence, also using \eqref{changeofmeas}, we see that, for any $\eps>0$ and all sufficiently large $N$,
$$
\begin{aligned}
{\tt P}&\Big(\big\|\smfrac 1N (jX_j)_{j\in\N}-\alpha^{\ssup\chi}\big\|_1>\eps\,\Big|\,{\mathfrak N}=N\Big)\\
&\leq{\tt P}_{\mu_N }\Big(\sum_{j=1}^N\big| j X_j-{\tt E}_{\mu_N}[jX_j]\big|>\frac \varepsilon 2 N\,\Big|\,{\mathfrak N}=N\Big)\\
&\leq C\sqrt N\frac 1{(\varepsilon N)^2}{\tt Var}_{\mu_N}\Big(\sum_{j=1}^N j X_j\Big)
\leq CN^{-3/2}{\tt Var}_{\mu_N}({\mathfrak N})\\
&\leq CN^{-1/2},
\end{aligned}
$$
where we used Lemma~\ref{lem:almost_gaussian} and the Chebychev inequality in the second step and Lemma~\ref{lem:E&Var}(1) in the final step.

Now we show the same assertion for the case $\chi=0$ with $ \alpha^{\ssup 0}=(1,0,0,\dots)$. For this, we show that ${\mathfrak N}^{\ssup 1}=X_1$ dominates the remaining particle number ${\mathfrak N}^{\ssup{\geq 2}}$, in the sense of
\begin{align}\label{eq:goal_small_a}
\Ptt_{\mu_N}({\mathfrak N}^{\ssup{\ge 2}}>\varepsilon N)\ll\Ptt_{\mu_N}({\mathfrak N}=N),\qquad N\to\infty,\quad \eps>0.
\end{align}
This will imply that 
\begin{equation}
    \Ptt_{\mu_N}\rk{{\mathfrak N}^{\ssup 1}\ge N(1-\e) \big|  {\mathfrak N}=N}= 1+o(1)\, ,
\end{equation}
i.e., almost all mass is in loops of length one, which implies the convergence of $\frac 1N (i X_i)_{i\in\N}$ towards $(1,0,0,\dots)$ under $\Ptt_{\mu_N}(\cdot|  {\mathfrak N}=N)$, and hence also under $\Ptt(\cdot|  {\mathfrak N}=N)$, due to \eqref{changeofmeas}.

We prove now \eqref{eq:goal_small_a}. Recall that we are in the case $\chi_N\to0$. For every fixed $\varepsilon>0$, by Chebyshev's inequality and \eqref{ENResti} for $R=2$, for large enough $N$,
\begin{align}
\Ptt_{\mu_N}({\mathfrak N}^{\ssup{\ge 2}}\ge\varepsilon N)
\le \frac{{\tt{Var}}_{\mu_N}({\mathfrak N}^{\ssup{\ge 2}})}{\rk{\varepsilon N-{\tt E}_{\mu_N}[{\mathfrak N}^{\ssup{\ge 2}}]}^2}
\le \frac 4{\varepsilon^2 N^2}{\tt{Var}}_{\mu_N}({\mathfrak N}^{\ssup{\ge 2}})
\leq o(\smfrac 1N),
\end{align}
where the last step follows from \eqref{eq:geR}.
This together with Lemma \ref{lem:almost_gaussian} proves \eqref{eq:goal_small_a}.
\end{proof}
\section{Identification of the free energy}\label{sec-freeenergy}

In this section, we prove the identification of the free energy in Theorem~\ref{thm-ODLROlongloops}(1)(d), respectively (2)(c).

 Recall $p_{a,N}(\mu)$ from \eqref{paNdef}. We abbreviate  $\Ptt=\Ptt_{\beta a_N,w/a_N}^{\ssup N}$ and $\Ptt_\mu=\Ptt_{\beta a_N,w/a_N,\mu}^{\ssup N}$. 
For any $N\in\N$  and $\mu\in(-\infty,0]$, we have from \eqref{change}
$$
Z_N(\beta,a_N,w)=\ex^{p_{a_N,N}(0)}\Ptt\rk{\Nrd=N}=\ex^{-\mu\beta a_N N+p_{a_N,N}(\mu)}\Ptt_{\mu}\rk{\Nrd=N}.
$$

    Assume first that $\chi>\rho_w$. In this case, set $\mu=0$. We then have that by Lemma \ref{lem-lowbounddenom} that
    \begin{equation}
        \Ptt\rk{\Nrd=N}=\exp\rk{-\l_1(W)\beta a_N^{\alpha/(2+\alpha)} N\rk{1-\smfrac{\rho_w}{\chi}} \rk{1+o(1)}}\, .
    \end{equation}
    Hence, we get that
    \begin{equation}
        {\rm f}_{\rm MF}(\beta,\chi)=\lim_{N\to\infty} \Big(-\frac{p_{a_N,N}(0)}{\beta N}+\lambda_1(W)a_N^{\alpha/(2+\alpha)}\rk{1-\smfrac{\rho_w}{\chi}}\Big)\, .
        \end{equation}
    Note that by the scaling, we have that
    \begin{equation}
    p_{a_N,N}(0)=a_N^{-d/2}p(0)\sim \frac{N}{\chi}p(0)\, ,
    \end{equation}
    and Theorem~\ref{thm-ODLROlongloops}(1)(d) follows.
    
    If $\chi<\rho_w$, we choose $\mu=\mu_N<0$ as in Lemma \ref{lem:b}. Lemma~\ref{lem:almost_gaussian} gives $\Ptt_{\mu}\rk{\Nrd=N}\asymp N^{-1/2}$. Hence, we can neglect this term and obtain
    \begin{equation}
     {\rm f}_{\rm MF}(\beta,\chi)=\lim_{N\to\infty} \Big(-\frac{p_{a_N,N}(a_N\mu_N)}{\beta N}+\mu_N a_N\Big)\, .
    \end{equation}
    In the case $\chi>0$, we again make the approximation $p_{a,N}(a_N\mu_N)\sim\frac{N}{\chi}p(a_N\mu_N)$, which implies  Theorem~\ref{thm-ODLROlongloops}(2)(c).
    
    If $\chi=0$, we approximate to first order 
    \begin{equation}
    p_{a,N}(a_N\mu_N)\sim\frac{N}{\chi_N}\frac{\ex^{\beta \mu_Na_N}}{(4\pi\beta)^{d/2}}\Wj_1\sim N\, ,
    \end{equation}
    which implies, via Lemma~\ref{lem:b} that $
    {\rm f}_{\rm MF}(\beta,0)=\lim_{N\to\infty}(\frac{1}{\beta}+\frac{\log(\chi_N)}{\beta^2})=-\infty.$

\bigskip
\noindent Tianyi Bai\\
Chinese Academy of Sciences\\
55 Zhongguancundong Rd, Haidian District, 100190 Beijing, China\\
{\tt tianyi.bai73@amss.ac.cn}
\medskip

\noindent Wolfgang König\\
Weierstrass Institute Berlin\\
Mohrenstraße 39, 10117 Berlin, Germany
\medskip

and
\medskip

\noindent TU Berlin, Institute for Mathematics\\
Straße des 17. Juni, 10623 Berlin, Germany\\
{\tt koenig@wias-berlin.de}\\
\medskip

\noindent Quirin Vogel\\
Ludwig Maximilian University of Munich, Mathematical Institute\\
Theresienstr. 39, 80333 München, Germany\\
{\tt quirinvogel@outlook.com}
\medskip

\end{document}